\documentclass{article}

\newlength{\myleftmargin}
\usepackage{amsmath}
\usepackage{amssymb}
\usepackage{amsthm}
\usepackage{color}
\usepackage{ifthen,changepage}
\DeclareSymbolFontAlphabet{\Bbb}{AMSb}

\newtheorem{theorem}{Theorem}[section]
\newtheorem{lemma}[theorem]{Lemma}
\newtheorem{proposition}[theorem]{Proposition}
\newtheorem{corollary}[theorem]{Corollary}
\newtheorem{definition}[theorem]{Definition}

\newcommand{\argmin}{\operatornamewithlimits{arg\,min}}
\newcommand{\Ysf}{\mathsf{Y}}

\newcommand{\atob}[2]{\emph{#1)} $\Rightarrow$ \emph{#2)}.}

\newlength{\fixboxwidthextra}
\newlength{\fixboxwidth}
\newcommand{\fix}[1]{\marginpar{\checkoddpage
\ifthenelse{\boolean{oddpage}}
{\hspace*{-0ex}\fbox{\parbox{\fixboxwidth}{\tiny #1}}}
{\hspace*{-\fixboxwidthextra}\fbox{\parbox{\fixboxwidth}{\tiny #1}}}}}

\definecolor{darkgreen}{rgb}{0,0.6,0}

\newcommand{\ca}[1]{{\mathcal #1}}

\newcommand{\mycdot}{\,\cdot\,}

\newcommand{\N}{\mathbb{N}}

\newcommand{\R}{\mathbb{R}}

\newcommand{\E}{\mathbb{E}}

\renewcommand{\a}{\alpha}
\renewcommand{\b}{\beta}
\newcommand{\g}{\gamma}
\newcommand{\G}{\Gamma}
\renewcommand{\d}{\delta}
\newcommand{\D}{\Delta}
\newcommand{\e}{\varepsilon}

\newcommand{\lb}{\lambda}

\newcommand{\s}{\sigma}

\newcommand{\p}{\varphi}
\renewcommand{\P}{\Phi}

\newcommand{\Om}{\Omega}

\DeclareMathOperator{\spann}{span}

\DeclareMathOperator{\id}{id}
\DeclareMathOperator{\cone}{cone}
\DeclareMathOperator{\im}{im}
\DeclareMathOperator{\myint}{int}
\DeclareMathOperator{\Gr}{graph}
\DeclareMathOperator{\Dom}{dom}

\newcommand{\eins}{\boldsymbol{1}}

\newcommand{\snorm}[1]{\Vert #1 \Vert}
\newcommand{\mnorm}[1]{\bigl\Vert \, #1 \, \bigr\Vert}

\newcommand{\hnorm}[1]{\Vert #1 \Vert_{H}}
\newcommand{\inorm}[1]{\Vert #1 \Vert_\infty}
\newcommand{\enorm}[1]{\Vert #1 \Vert_{E}}
\newcommand{\fnorm}[1]{\Vert #1 \Vert_{F}}

\newcommand{\Lx}[2]{L_{#1}(#2)}

\newcommand{\SF}{S}

\newcommand{\xstar}{x_\star}

\newcommand{\relint}[2]{\mathring #1^#2}

\usepackage{enumitem}

\usepackage{amsmath}
\usepackage{amssymb}
\usepackage{color}
\usepackage{url}
\usepackage{ifthen,changepage}
\DeclareSymbolFontAlphabet{\Bbb}{AMSb}

\title{Representation of Quasi-Monotone Functionals by Families of Separating Hyperplanes}

\author{Ingo Steinwart\\
Institute for Stochastics and Applications\\
Faculty 8: Mathematics and Physics\\
University of Stuttgart\\
D-70569 Stuttgart Germany \\
\texttt{\small ingo.steinwart@mathematik.uni-stuttgart.de}
}

\begin{document}
\maketitle

\begin{abstract}
We characterize when the level sets of a continuous quasi-monotone functional defined on a suitable convex subset of a 
normed space can be uniquely represented by a family of bounded continuous functionals. Furthermore, we investigate 
how regularly these functionals depend on the parameterizing level.
Finally, we show how this question relates to the recent problem of property elicitation
that simultaneously attracted interest in machine learning, statistical evaluation of forecasts, and finance.
\end{abstract}

\section{Introduction}\label{sec:intro}

Suppose we have a normed space $(E, \enorm\cdot)$,  a non-empty convex subset $B\subset E$ that is contained in some 
closed affine hyperplane
not passing the origin,
and a continuous, in general non-linear, functional 
$\G:B\to \R$ for which
the level sets $\{\G=r\} := \{x\in B: \G(x) = r\}$ are convex for all $r\in \im \G$.
Let us denote the interior of the image   of $\G$ by $I$, that is $I:= \mathring{\G(B)} = \mathring{\im \G}$.
In this paper we consider the following questions:
\begin{enumerate}
  \item Under which conditions is there a unique 
  family $(z'_r)_{r\in I}$ of (normalized) bounded linear functionals on $E$ 
  such that for all $r\in I$ we have 
  \begin{align*}
 \{\G<r\} &= \{z_r'< 0\} \cap B \\
\{\G=r\} &= \{z_r'= 0\} \cap B\\ 
\{\G>r\} &= \{z_r'>0\} \cap B\, ?
\end{align*}
  \item When is the map $r\mapsto \hat z'_r$ measurable or even continuous?
\end{enumerate}
While at first glance these questions seem to be of little practical value they actually lie at the heart of 
a problem that recently attracted interest in machine learning, statistical evaluation of forecasts, and finance,
see \cite{StPaWiZh14a,AgAg15a,FrKa15a}, \cite{GnRa07a,Gneiting11a}, and \cite{LaPeSh08a,FrKa14a, Ziegel14a,WaZi15a}, respectively,
as well as the various references mentioned in these articles.

Let us 
 briefly explain this problem
 while generously ignoring all mathematical issues. 
 To this end,
   let 
$\ca P$ be a set of probability
measures on $\Omega$,  and 
$\G \colon \ca P \to \R$ be an arbitrary map, which in the following will be called a property on $\ca P$.
 Simple examples of  properties of distributions on $\Om=\R$ are 
the mean, the median, and the variance, while  more complicated properties are the (conditional) value at risk
and conditional tail expectation. Now, for some properties 
including the mean, the median, and others, see  \cite{Gneiting11a} for an extensive list,
there exists a so-called scoring function 
$\SF: \Om\times   \R\to \R$ such that 
\begin{equation}\label{def-consist}
 \G(P) = \underset{r \in \R}{\argmin} \ \mathbb{E}_{\Ysf\sim P}S(r,\Ysf)
\end{equation}
for all $P\in \ca P$, i.e.~$\G(P)$ is the unique minimizer of the expected scoring function.
Such properties, which are  called elicitable, have various positive aspects: For example,
if $P$ is only approximately known, e.g.~by data, then we can replace $P$ by its approximation $\hat P$ in 
\eqref{def-consist} to estimate $\G(P)$ by $\G(\hat P)$. Similarly, if we have two estimates $\hat r_1$ and 
$\hat r_2$ of $\G(P)$ then we can compare these by comparing the corresponding values 
$\mathbb{E}_{\Ysf\sim P}S(\hat r_1,\Ysf)$ and $\mathbb{E}_{\Ysf\sim P}S(\hat r_2,\Ysf)$, or 
their $\hat P$-approximations if $P$ is unknown, see e.g.~\cite{Steinwart07a,LambertXXa}.
While these observations are rather straightforward they lie, in a conditional i.e.~functional 
form, at the very core of a huge class of 
machine learning algorithms, namely so-called (regularized) empirical 
risk minimizers \cite{Vapnik98,ScSm02,StCh08}.

Elicitable properties are therefore highly desirable, but unfortunately, not every property is elicitable.
Indeed, \cite{Osband85a}, see also \cite{LaPeSh08a,Gneiting11a}
showed that for convex $\ca P$ an elicitable property needs to have 
convex level sets, and the variance does, for example, not have such level sets.
Having convex level sets alone is, however, not sufficient for elicitability, 
 and hence 
one needs additional assumptions to obtain sufficient conditions. To find such conditions, one key idea, 
known as Osband's principle \cite{Osband85a,LaPeSh08a,Gneiting11a,StPaWiZh14a},
is to take 
the derivative on the right-hand side of \eqref{def-consist} to (hopefully) find
that $\G(P)$ can be characterized as the only zero of the function $r\mapsto \mathbb{E}_{\Ysf\sim P}S'(r,\Ysf)$.
Now observe that the linearity of $\E$ in $P$ makes it possible to write 
\begin{align}\label{osband}
   \mathbb{E}_{\Ysf\sim P}S'(r,\Ysf) = \bigl\langle S'(r,\cdot), P\bigr\rangle \, ,
\end{align}
where $\langle z, p\rangle := z(p)$ denotes the evaluation of a linear
functional $z$ at some vector $p$. Clearly, if $z_r' := S'(r,\cdot)$ is interpreted as a functional over 
$\spann \ca P$ (or its closure), and if these functionals satisfy the equations in our first question,
then $\G(P)$ can indeed be characterized as the only zero of the function above. Consequently,  this 
part of  Osband's principle  will work as soon as we have a positive answer to our first question.
However, to construct one (or actually all) scoring functions 
from $S'$ one needs additional regularity of $r\mapsto S'(r,\cdot)$ such as suitable 
measurability or even continuity, see \cite{LambertXXa,StPaWiZh14a}.
This motivates the second question we deal with in this paper.

We also like to note that two questions above have a simple answer, if $\G$ is defined on the entire space $E$
by
\begin{align}\label{full-G}
   \G(x) := h(\langle z', x\rangle)\, , \qquad \qquad x\in E\, ,
\end{align}
where $z':E\to \R$ is a bounded linear functional and $h:\R\to \R$ is a strictly monotone map.
Indeed, if we pick an $\p'\in E'$ with $B\subset \{\p'=1\}$ and assume, for example, that $h$ is strictly increasing, then 
\begin{displaymath}
   z_r' = z' - h^{-1}(r) \p'
\end{displaymath}
defines such a family of separating functionals. However, if $B$ is 'too small', then various $\p'$ are possible,
and therefore this construction is, even after renormalization, not unique.
Nonetheless, \eqref{full-G} is somewhat archetypal, since
 for finite dimensional spaces $E$,
every continuous $\G:E\to \R$ that has convex level sets is of the form \eqref{full-G}
for some monotone $h$, see
\cite{ThPa73a} in combination with  Lemma \ref{quasi-monotone}. Moreover, without $h$ being strictly monotone,
we cannot expect a positive answer to our first question, 
and hence this assumption in \eqref{full-G} was not a restriction, either.
In general, however, continuous $\G:B\to \R$ with convex level sets are \emph{not} of the form \eqref{full-G},
not even in three dimensions, since roughly speaking the form \eqref{full-G} is forced by the requirement that 
the level sets  cannot intersect, and hence they need to be parallel if $\G$ is defined on the full space $E$.
But for smaller $B$, this is no longer necessary, and it is actually elementary to construct such examples.

The rest of this work is organized as follows: In Section \ref{sec:exist} we characterize when we have a separating
family in the sense of the first question. Section \ref{sec:meas-sep} then investigates measurable dependence on $r$
and Section \ref{sec:cont-sep} deals with continuous dependence. In Section \ref{sec:quasi-monotone}
we present some auxiliary results on quasi-monotone functions and all proofs can be found in Sections \ref{sec:proof-1}
to \ref{sec:proof-3}.

\section{Existence and Uniqueness of the Separating Family}\label{sec:exist}

In this section we give positive answers to the first  question raised in the introduction, that is, we show that
under some conditions  on $\G$ and $B$ specified below there 
exists a unique family of separating bounded linear functionals

Let us begin by  fixing some notations.
Throughout this paper  $(E, \enorm \cdot)$ is a
normed space if not stated otherwise, $E'$ denotes its dual and $B_E$  its closed unit ball. 
 Moreover, 
for an $A\subset E$ we write $\relint AE$ for the interior of $A$ with respect to the norm $\enorm\cdot$.
If this norm is known from the context we may abbreviate notations by $\mathring A:= \relint AE$, and 
for typesetting reasons, we sometimes also write $\myint A:=  \mathring A$. Similarly, $\overline A^{_E}$
denotes the closure of $A$ with respect to the norm $\enorm\cdot$, and if the latter is known from the context 
we may again write $\overline A$.
Moreover, $\spann A$ denotes the linear space spanned by $A$ and $\cone A:= \{\a x: \a\geq 0, x\in A \}$ 
denotes the cone generated by $A$. In addition, the null space of a linear functional $z':E\to \R$
is denoted by $\ker z$, and the restriction of a function $f:A\to B$ onto $C\subset A$ is denoted by $f_{|C}$.

With the help of this notations we can now formulate our first set of assumptions that describe 
the set $B$. Throughout these assumptions, $E$ denotes a normed space, $B\subset E$ is non-empty and convex, and 
$H:=\spann B$.

\begin{description}[itemsep=-.5ex]
\item[B1 (Simplex face).] 
	There exists a $\p'\in E'$
	such that $B\subset \{\p'=1\}$. 

\item[B2 (Dominating norm).] There exists an $\xstar\in B$ such that for  $A:= -\xstar + B$ and
	\begin{displaymath}
	 F:= \spann A 
	\end{displaymath}
	  there exists a norm $\fnorm \cdot$ on $F$ with $\enorm \cdot \leq \fnorm\cdot$.
	 
\item[B2* (Non-empty relative interior).] Assumption \textbf{B2} is satisfied and  $ 0 \in \relint AF$. 
	
\item[B3 (Cone decomposition).] There exists a constant $K>0$ such that for all $z\in H$
	there exist $z^-, z^+\in \cone B$ with  $z=z^+-z^-$ and 
	\begin{displaymath}
	\enorm{z^-} + \enorm{z^+} \leq K \enorm z\, .
	\end{displaymath}

\item[B4 (Denseness).] The space  $H$ is dense in $E$ with respect to  $\enorm\cdot$. 

\end{description}

To illustrate these assumptions 
in view of the elicitation question raised in the introduction,
we fix a probability measure $\mu$ on  
some measurable space $(\Om, \ca A)$, and 
 consider 
the set of bounded, integrable probability densities with respect to $\mu$, that is
\begin{equation}\label{Dgeq0}
   \D^{\geq 0}  :=   \{  h\in \Lx \infty \mu \colon   h\geq 0,\, \E_\mu h =1  \}\, .
\end{equation}
Our set $\ca P$ will then be $\ca P:= \{hd\mu: h\in  \D^{\geq 0}\}$.
For $p\in [1,\infty)$,  $E:= \Lx {p} \mu$, and $\p' := \E_\mu(\cdot)$ 
we then verify that $ \D^{\geq 0}$  satisfies \textbf{B1}, and for $p=1$,
the norm induced on $\ca P$ equals the total variation  norm. 
 Moreover, we have $H= \Lx \infty \mu$ and therefore \textbf{B4} is obviously satisfied.
 Furthermore, by considering $h=\max\{0, h\}- \max\{0, -h\}$ we obtain \textbf{B3} for 
 $K:= 2^{1-1/p}$.
 Consequently, the only task left is to find 
 a suitable $h_\star\in \D^{\geq 0}$ and 
 an appropriate norm $\fnorm\cdot$.
 Unfortunately, taking $\fnorm \cdot = \enorm\cdot $ won't work in this example, since 
 the elements in 
 \begin{displaymath}
    -h_\star +  \D^{\geq 0} =  \{  h\in \Lx \infty \mu \colon   h\geq -h_\star,\, \E_\mu h =0  \}
 \end{displaymath}
are pointwise bounded from below  by $-h_\star$
 but this cannot be guaranteed in any $\enorm\cdot$-ball in $F$ around the origin. However, for
 $\fnorm\cdot := \inorm \cdot$ and $h_\star :=  \eins_\Om$  
 Assumption \textbf{B2*} does hold.
 
The example above illustrates, that the choice of $\fnorm\cdot$ may give some extra freedom when applying 
the results of this paper. Unfortunately, however, this   freedom comes   
for an extra price we have to pay at a different condition. Before we can explain the details let us 
present the following lemma 
 that investigates the spaces $H$ and $F$ in a bit more detail.

\begin{lemma}\label{F-is-in-kernel}
Let \textbf{B1} and \textbf{B2} be satisfied.
Then, the space $F$ satisfies  $F\subset \ker \p'$. In particular, we have $\xstar \not\in F$
and 
\begin{displaymath}
 H=  F\oplus \R\xstar \, .
\end{displaymath}
Furthermore, if we equip $H$  with the norm $\hnorm\cdot$, defined by
\begin{displaymath}
 \hnorm{y+\a \xstar} := \fnorm y + \enorm{\a \xstar}\, , \qquad \qquad \qquad y\in F, \a\in \R,
\end{displaymath}
then, 
we have $\enorm\cdot \leq \hnorm\cdot$ on $H$, $\fnorm \cdot = \hnorm\cdot$ on $F$.
Finally, for all $x_1,x_2\in B$ we have $x_1-x_2\in F$.
\end{lemma}

Roughly speaking, Lemma \ref{F-is-in-kernel} provides a simple way to extend the norm $\fnorm\cdot$
to the space $H=\spann B$ in which most of our initial geometric arguments take place. In addition,
it is a key ingredient in the second of the following set of assumptions on $\G$.
Throughout these assumptions $B\subset E$ again denotes a non-empty convex subset of the normed space $E$.
Moreover, $\G:B\to \R$ denotes an arbitrary map and we write $I:= \mathring{\G(B)}$. Finally, we
 assume that \textbf{B1} and \textbf{B2} are satisfied whenever this is necessary.

\begin{description}[itemsep=-.5ex]

\item[G1 ($F$-continuous and convex level sets).] The map 
	$\G:B\to \R$ is  $\fnorm\cdot$-continuous and its level sets 
	 $\{\G=r\}$ are convex for all $r\in \im \G$.

\item[G1* ($E$-continuous and convex level sets).] The map 
	$\G:B\to \R$ is  $\enorm\cdot$-continuous and its level sets 
	 $\{\G=r\}$ are convex for all $r\in \im \G$.
	
\item[G2 (Locally non-constant).] 
    For all $r\in I$, $\e>0$, and $x\in \{\G=r\}$, there exist $x^-\in \{\G<r\}$ and 
    $x^+\in \{\G>r\}$ such that $\hnorm{x-x^-}\leq \e$ and $\hnorm{x-x^+}\leq \e$.

\item[G3 (Locally non-constant continuous extension).] 
    We have a $\enorm\cdot$-continuous extension $\hat \G: \overline B\to \R$  of $\G$ such that 
    for all $r\in I$, $\e>0$, and $x\in \{\hat \G=r\}$, there exist $x^-\in \{\hat \G<r\}$ and 
    $x^+\in \{\hat \G>r\}$ with $\enorm{x-x^-}\leq \e$ and $\enorm{x-x^+}\leq \e$.    
    
\end{description}

By Lemma \ref{F-is-in-kernel} we know that 
for all $x_1,x_2\in B$ we have $x_1-x_2\in F$ and thus 
$\fnorm{x_1-x_2} = \hnorm{x_1-x_2}$. Consequently, the assumed $\fnorm\cdot$-continuity in \textbf{G1} is 
well-defined and 
equivalent to  $\hnorm\cdot$-continuity. Moreover, if \textbf{B1} and \textbf{B2} are satisfied, then 
$\fnorm\cdot$ dominates $\enorm\cdot$, and therefore \textbf{G1*} implies \textbf{G1} in this case.

At first glance the convexity of the level sets  and the continuity of $\G$ are conceptually  simple 
assumptions.
When combined, however, they have a significant impact on the shape of $\G$ and its level sets.
To illustrate this 
 let us 
 recall that a function $\G:B\to
\R$ defined on some convex subset $B\subset E$ of a vector space $E$ is
called \emph{quasi-convex}, if, for all $r\in \R$, 
the \emph{sublevel sets} $\{\G\leq r\}$ are convex.  
It is well-known, see \cite{GrPi71a} for some historic remarks,
and also a simple exercise 
that $\G$ is quasi-convex, if and only if
\begin{equation}\label{def-quasi-conv}
  \G\bigl((1 - \a) x + \a y\bigr) \leq \max \bigl\{\G(x),\G(y)\bigr\}
\end{equation}
holds for all $x,y\in B$ and $\a\in [0,1]$. We further say that $\G$ is \emph{strictly quasi-convex},
if, in addition, this inequality is strict for all $x,y\in B$ with $\G(x) \neq \G(y)$ and all $\a\in (0,1)$. 
Analogously, $\G$ is called
\emph{(strictly) quasi-concave}, if $-\G$ is (strictly) quasi-convex.  Finally, $\G$ is called
\emph{(strictly) quasi-monotone}, if $\G$ is both (strictly) quasi-convex and (strictly) 
quasi-concave. It can be shown, that 
$\G$ is quasi-monotone, if and only if $\G$ is monotone on each segment, see Lemma \ref{quasi-monotone},
and if $\G$ is continuous, quasi-monotonicity is also equivalent to the convexity of all level sets, see Lemma \ref{quasi-m-char}.
Consequently, if \textbf{G1} or \textbf{G1*} is satisfied, then its  level sets cannot, for example, 
form an alveolar partition of $B$ or a triangulation partition,
since both would contradict the convexity of the sublevel sets.
We refer to \cite{LambertXXa} for some nice illustrations.
\emph{Without} the continuity, however, such  partitions would be perfectly fine.

Assumption \textbf{G2} essentially states that $\G$ is not constant on arbitrarily small balls $B\cap \e B_H$, where the used 
norm $\hnorm\cdot$ is typically larger than $\enorm\cdot$, that is, the considered balls $\e B_H$ are smaller than 
the balls $\e B_E$. In particular, if a larger norm $\fnorm \cdot$ is required to ensure \textbf{B2*}, then 
in turn this choice  leads to a stronger version of \textbf{G2}.

Finally, \textbf{G3} will be used to extend results for $B$ to $\overline B$.
This will particularly useful if the set $B$ is only an auxiliary set in the sense that we are actually 
interested in $\overline B$, instead. For example, in \eqref{Dgeq0} we only considered bounded densities
to ensure \textbf{B2*}. In general, however, one might be interested in all probability densities, that is, in the set 
$\overline { \D^{\geq 0}}$. Now \textbf{G3} essentially states that if we actually have a continuous functional 
on $\overline B$ then we need a weak version of \textbf{G2} on $\overline B\setminus B$, too.

Before we present our first main results, let us finally  introduce the following definition, which formally describes 
the functionals   we seek.

\begin{definition}
   Let $E$ be a normed space, $\G:B\to \R$ be a map and $I:= \mathring{\G(B)}$.
   Then, a family $(z_r')_{r\in I}$ of linear maps $z_r':E\to \R$ is called a separating family for $\G$,
   if for all $r\in I$ we have 
   \begin{align}\label{unique-separation-a1}
 \{\G<r\} &= \{z_r'< 0\} \cap B \\ \label{unique-separation-a2}
\{\G=r\} &= \{z_r'= 0\} \cap B\\ \label{unique-separation-a3}
\{\G>r\} &= \{z_r'>0\} \cap B\, .
\end{align}
\end{definition}

Note that in the definition above the maps $z_r'$ are not necessarily continuous. In the following, however, 
all obtained separating families will consist of continuous functionals, but depending on the situation, 
the continuity will be with respect to either $\hnorm\cdot$ or $\enorm\cdot$.

The following result characterizes the existence of a separating family in $H'$.

\begin{theorem}\label{unique-separation-char-1}
Let \textbf{B1} and \textbf{B2*} be satisfied and 
 $\G:B\to \R$ be a map. We write $I:= \mathring{\G(B)}$ and $B_0 := \G^{-1}(I)$. Then the following statements are equivalent:
\begin{enumerate}
   \item Assumptions \textbf{G1} and \textbf{G2} are satisfied.
   \item The map $\G$ is $\fnorm\cdot$-continuous and quasi-monotone. Moreover, $\G_{|B_0}$ is strictly quasi-monotone.
   \item There exists a  separating family $(z_r')_{r\in I} \subset H'$ for $\G$.
   \item There exists a unique separating family $(z_r')_{r\in I} \subset H'$ for $\G$ with $\snorm{z_r'}_{H'} = 1$ for all $r\in I$.
\end{enumerate}
\end{theorem}

Theorem \ref{unique-separation-char-1} shows  that under the  assumptions \textbf{B1} and \textbf{B2*}
on $B$, the conditions \textbf{G1} and \textbf{G2} on $\G$ are both necessary and sufficient for 
the existence of a separating family  in $H'$. In addition, it shows that the only freedom for choosing this family is
the scaling of its members. In combination with Lemma \ref{quasi-m-char} we 
finally see that \textbf{G2} can be replaced by the \emph{norm-independent} 
strict quasi-monotonicity of $\G_{|B_0}$.

Our next goal is to present a similar characterization for separating functionals that are $\enorm\cdot$-continuous. 
To this end, we write 
 $\snorm{z'}_{E'} = 1$ for the norm
of a functional $z'\in (H, \enorm\cdot)'$. 

\begin{theorem}\label{unique-separation-char-2}
Let \textbf{B1}, \textbf{B2*}, \textbf{B3} be satisfied and 
 $\G:B\to \R$ be a map. We write $I:= \mathring{\G(B)}$ and $B_0 := \G^{-1}(I)$. Then the following statements are equivalent:
\begin{enumerate}
   \item Assumptions \textbf{G1*} and \textbf{G2} are satisfied.
   \item The map $\G$ is $\enorm\cdot$-continuous and quasi-monotone. Moreover, $\G_{|B_0}$ is strictly quasi-monotone.
   \item There exists a  separating family $(z_r')_{r\in I} \subset (H, \enorm\cdot)'$ for $\G$.
   \item There exists a unique separating family $(z_r')_{r\in I} \subset (H, \enorm\cdot)'$ for $\G$ with $\snorm{z_r'}_{E'} = 1$ for all $r\in I$.
\end{enumerate}
Moreover, if  condition iv) is true and 
 \textbf{B4} is also satisfied, then, for all $r\in I$,  
there exists exactly one $\hat z_r'\in E'$ such that $(\hat z_r')_{|H} = z_r'$
and  $\snorm{\hat z_r'}_{E'} = 1$.
\end{theorem}

When we apply Theorem \ref{unique-separation-char-1} to the example discussed around \eqref{Dgeq0},
we see that there is a (unique) family of separating hyperplanes
$(\hat z_r')_{r\in I}\subset \Lx{p'}\mu$
for $\G$ \emph{if and only if} our property $\G$
is  $\snorm\cdot_p$-continuous, quasi-monotone, and even strictly quasi-monotone on $B_0$.
This is, to the best of our knowledge, the first characterization when the part around \eqref{osband}
of Osband's principle does work.
Nonetheless,
we like to mention that the implications $i) \Rightarrow iv)$ of 
Theorems \ref{unique-separation-char-1} and \ref{unique-separation-char-2}  have already been shown in the 
unreviewed appendix of \cite{StPaWiZh14a}. However, the remaining implications are new and so is the following 
third and last result in this section.

\begin{theorem}\label{unique-extended-separation}
 Assume that \textbf{B1}, \textbf{B2*}, \textbf{B3}, \textbf{B4}, \textbf{G1*}, \textbf{G2}, and \textbf{G3} are satisfied, and let 
$(\hat z_r')_{r\in I}\subset E'$ be the separating family   found in Theorem \ref{unique-separation-char-2}.
Then, for all $r\in I$, we have
\begin{align*}
 \{\hat \G<r\} &= \{\hat z_r'< 0\} \cap \overline B \\
\{\hat \G=r\} &= \{\hat z_r'= 0\} \cap \overline B\\ 
\{\hat \G>r\} &= \{\hat z_r'>0\} \cap \overline B\, .
\end{align*}
\end{theorem}

Theorem \ref{unique-extended-separation} essentially shows that a separating family 
for $\G$  in $E'$
is also a separating family for a continuous extension $\hat \G$ satisfying \textbf{G3}. Here we note that 
\textbf{G3} can again be replaced by a strict quasi-convexity assumption. Moreover,
a family satisfying the three equalities in Theorem \ref{unique-extended-separation} 
is a separating family of $\G$, and therefore, the implications of Theorem \ref{unique-separation-char-2} apply.
In particular, if such a family exists, then 
$\G$ needs to be $\enorm\cdot$-continuous and quasi-monotone, and $\G_{|B_0}$ needs to be strictly quasi-monotone.
Moreover, by repeating the arguments used in the proof of Theorem \ref{unique-separation-char-1}
we see that even $\hat \G$ needs to be $\enorm\cdot$-continuous and quasi-monotone.

Finally, let us again have a quick look at the example discussed around \eqref{Dgeq0}. 
Here we see that the part around \eqref{osband}
of Osband's principle works, if, for example, $p=1$ and $\G$ is a property on the set of all $\mu$-absolutely 
continuous probability measures that is continuous with respect to the total variation norm and satisfies the 
(strict) quasi-monotonicity assumptions discussed above. As far as we know, this is the first such result for 
probability measures not having a bounded density.

\section{Measurable Dependence of the Separating Hyperplanes}\label{sec:meas-sep}

In the previous section we have see that under some 
conditions on both $B$ and $\G$ we have a unique family of separating hyperplanes $(\hat z_r')_{r\in I}\subset E'$.
Our goal in this section is to investigate under which supplemental  assumptions the resulting 
map $r\mapsto \hat z_r'$ is measurable.

To this end, we will, 
consider the 
following two additional assumptions:
\begin{description}[itemsep=-.5ex]
\item[B5 (Completeness and separable dual).] The space $E$ is a Banach space and its dual $E'$ is separable.

\item[G4 (Measurability).]  The pre-image $B_0:= \G^{-1}(I)$ is a  Borel measurable subset of $E$. 
\end{description}

Before we discuss these assumptions, we like to present the main result of this section, which
shows that the map $r\mapsto \hat z'_r$ is measurable provided that \textbf{B5} and \textbf{G4} hold.
%
To formulate it, we write $\ca B(X)$ for the Borel $\s$-algebra of a given topological space $X$.
Moreover, we equip the interval $I$ with the Lebesgue
 completion $\hat{\ca B}(I)$ of its Borel $\s$-algebra $\ca B(I)$.

\begin{theorem}\label{strong-measurable-and-int}
Assume that \textbf{B1} to \textbf{B5}, as well as \textbf{G1*}, \textbf{G2}, and  \textbf{G4} are satisfied.
Then, the map 
$Z:I\to E'$ defined by 
\begin{displaymath}
 Z(r):=  \hat z_r' \, ,
\end{displaymath}
where $\hat z_r'\in E'$ are the unique functionals obtained in Theorem \ref{unique-separation-char-2},
is measurable with respect to the $\s$-algebras $\hat{\ca B}(I)$ and $\ca B(E')$, 
and it is also 
an $E'$-valued 
measurable function in the sense of Bochner 
integration theory with respect to the $\s$-algebra $\hat{\ca B}(I)$.
\end{theorem}

Let us briefly return to our initial example \eqref{Dgeq0} of bounded probability densities.
There it can be shown that \textbf{G4} is automatically satisfied, and \textbf{B5}
is satisfied if and only if $1<p<\infty$ and $E=\Lx p \mu$ is separable.
If the remaining assumptions of Theorem \ref{strong-measurable-and-int} hold true, too, we thus  see
that that map $Z:I\to \Lx {p'} \mu$ is measurable.
Unfortunately, however, this may not be the desired property. Indeed, in \eqref{Dgeq0} it seems natural to take $p=1$
and ask for the measurability of $Z:I\to \Lx {\infty} \mu$.
Clearly, \textbf{B5} is violated in this case, and thus Theorem \ref{strong-measurable-and-int}
does not provide the desired answer. The following corollary partially addresses this issue.

\begin{corollary}\label{meas-cor}
Assume that   \textbf{B1}, \textbf{B2*}, \textbf{B3}, \textbf{B4}, \textbf{G1*} and \textbf{G2}
are satisfied for $E$, $B\subset E$, $\p'\in E'$, $F$, and $\G:B\to \R$ and let 
$(\hat z_r')_{r\in I}\subset E'$ be the corresponding
family of separating functionals found in Theorem \ref{unique-separation-char-2}.
In addition, let $E_0\hookrightarrow E$ be a continuously embedded Banach space with $B\subset E_0$ such that 
 \textbf{B2} to \textbf{B5}, as well as \textbf{G1*}, and  \textbf{G4} are satisfied for $E_0$ and  $F$.
Then we also obtain a family $(\hat z_{0,r}')_{r\in I}\subset E_0'$ of separating functionals by 
Theorem \ref{unique-separation-char-2} and this family is measurable in the sense of Theorem \ref{strong-measurable-and-int}
with respect to the space $E_0'$. Moreover, there exists a measurable map $\a:(I,\hat {\ca B}(I))\to (\R, \ca B(\R))$ such that 
for all $r\in I$ we have 
$\a(r)>0$  and 
\begin{align}\label{meas-cor-a1}
   (\hat z_r')_{|E_0} = \a(r) \hat z_{0,r}'\, .
\end{align}
\end{corollary}

Note that the functionals $\hat z_r'$ and $\hat z_{0,r}$ are normalized with respect to the dual norms of 
$\enorm\cdot$ and $\snorm\cdot_{E_0}$, respectively, and therefore, we typically have $\a(r) \neq 1$.
Moreover, the assumptions ensure that $E_0$ is dense in $E$ and therefore $\hat z_{0,r}'$ can be uniquely extended 
to a continuous functional on $E$, namely to  $\frac 1 {\a(r)} \hat z_r'$.

The main message of 
Corollary \ref{meas-cor} is  that $r\mapsto  (\hat z_r')_{|E_0}$ is measurable with respect to $E_0'$,
that is, even if we use the normalization with respect to $E$, we still obtain measurability with respect to  $E_0'$.
Applied to our motivating example in front of Corollary \ref{meas-cor}, this means that we obtain 
a family $(h_r)_{r\in I}\subset \Lx \infty \mu$ that represent 
the functionals $\hat z_r'\in (\Lx 1 \mu)'$ such that $r\mapsto h_r$ is measurable with 
respect to  $\Lx p \mu$ for all $p\in (1,\infty)$. The latter can then be used to conclude that we find a 'version'
$\tilde h:I\times \Om\to \R$ of this family that is $\hat{\ca B}(I) \otimes \ca A$-measurable, and in turn
this measurability can be used to make the second part of Osband's principle work, see \cite{StPaWiZh14a} 
for a more elementary but also  technically more involved approach.
Moreover, our normalization in $(\Lx 1 \mu)' = \Lx \infty \mu$ means that we have 
 $\inorm{h_r}=1$ for all $r\in I$, and 
the latter is the additional information provided by Corollary \ref{meas-cor}
when compared to Theorem \ref{strong-measurable-and-int}.
Finally, whether  $r\mapsto h_r$ is actually measurable with 
respect to  $\Lx \infty \mu$ remains an open question.

\section{Continuous Dependence of the Separating Hyperplanes}\label{sec:cont-sep}

In this section we investigate even stronger regularity of $r\mapsto \hat z_r'$, namely some
forms of continuity.
To this end, we need the following two  
additional assumptions:
\begin{description}[itemsep=-.5ex]
\item[B6 (Separable Banach space).] The space $E$ is a separable Banach space.

\item[G5 (Weak level set continuity).]  For all $r\in I$ and all sequences $(r_n)\subset I$ with $r_n\to r$ there exists 
 an $x \in H \setminus \spann\{\G=r  \}$ such that 
 \begin{align}\label{g5-conv}
  d\bigl(x, \spann\{\G=r_n  \}  \bigr) \to  d\bigl(x, \spann\{\G=r  \}  \bigr)\, ,
 \end{align}
 where the distance is measured in the norm $\enorm\cdot$.
\end{description}

Note that the separability of $E$ is not really necessary if one works with nets instead of sequences throughout the
proofs for this section. However, 
for the sake of simplicity, we decided to stick with sequences.
Also note that \textbf{G5} essentially means, see 
the proof of Theorem \ref{cont-dep} for details, that $\langle \hat z_{r_n}', x\rangle \to
\langle \hat z_{r}', x\rangle$ for this particular $x$. In other words, \textbf{G5}
asserts that there is at least one $x\not \in \ker z_r'$ for which we have some very weak sort of
'continuity'. Here we put continuity in quotation marks since unlike in continuity,
\textbf{G5} allows $x$ to depend on the chosen sequence $(r_n)$.

The following result shows that this is already enough to obtain weak*-continuity of 
$r\mapsto \hat z_r'$.

\begin{theorem}\label{cont-dep}
   Let \textbf{B1}, \textbf{B2*}, \textbf{B3}, \textbf{B4}, \textbf{B6}, \textbf{G1*}, \textbf{G2}
   be satisfied. Moreover, let $Z:I\to E'$ be defined by 
\begin{displaymath}
 Z(r):=  \hat z_r' \, ,
\end{displaymath}
where $\hat z_r'\in E'$ are the unique functionals obtained in Theorem \ref{unique-separation-char-2}.
   Then the following statements are equivalent:
   \begin{enumerate}
    \item Assumption \textbf{G5} be satisfied.
    \item For all  $r\in I$, all sequences $(r_n)\subset I$ with $r_n\to r$, and all $x \in E$ 
    convergence \eqref{g5-conv}
    holds.
    \item  For all  $r\in I$, all sequences $(r_n)\subset I$ with $r_n\to r$, and all $x \in E$ we have 
    $\langle \hat z_{r_n}', x\rangle \to \langle \hat z_r', x\rangle$.
   \end{enumerate}
%
%
\end{theorem}

If $E'$ is a uniformly convex Banach space and the assumptions of Theorem \ref{cont-dep} are satisfied, then 
the map
$Z:I\to E'$
is actually norm continuous. Indeed, uniformly convex Banach spaces are reflexive, see \cite[Prop.~1.e.3]{LiTz79}
or \cite[p.~196]{Beauzamy82}, 
and thus weak*-continuity equals weak-continuity. Moreover, our normalization guarantees $\snorm{\hat z_r'} = 1$
for all $r\in I$, and therefore, we obtain norm-continuity by \cite[p.~198]{Beauzamy82}.

The next result shows that \textbf{G5} is superfluous, even for norm-continuity,  
as long as $E$ is finite-dimensional. 
In a  different form it has also been shown in \cite{LambertXXa}.

\begin{corollary}\label{cont-dep-fin-dim}
   Let \textbf{B1}, \textbf{B2*}, \textbf{B3}, \textbf{B4},   \textbf{G1*}, \textbf{G2}
   and   be satisfied, and $E$ be finite dimensional.
   Then, the map 
$Z:I\to E'$ defined by 
\begin{displaymath}
 Z(r):=  \hat z_r' \, ,
\end{displaymath}
where $\hat z_r'\in E'$ are the unique functionals obtained in Theorem \ref{unique-separation-char-2},
is norm continuous.
\end{corollary}

Note that for finite dimensional spaces $E$, condition \textbf{B4} reduces to $H=E$, that is $E=\spann B$.
Moreover, 
in the case of the example discussed around \eqref{Dgeq0} a finite dimension of $E$ means that $\Om$ is finite.

Finally, note that condition \textbf{G5} does not appear
in
Corollary \ref{cont-dep-fin-dim}. Since \textbf{B6} is automatically satisfied for finite dimensional spaces,
we thus conclude by Theorem \ref{cont-dep} that  \textbf{G5} always holds in this setting. Whether this 
is true in more general settings remains an open question.

\section{Quasi-Monotonicity}\label{sec:quasi-monotone}

In this section we briefly recall some simple facts about quasi-monotone functions we need 
throughout the paper. Some of these results may be folklore but since we were not able 
to find references establishing these results in the needed generality, we added their proofs.

We begin with the following characterization of quasi-monotonicity.

\begin{lemma}\label{quasi-monotone}
   Let $E$ be a vector space, $X\subset E$ be a convex subset and $\G:X\to \R$ be a function. Then 
   the following statements are equivalent:
   \begin{enumerate}
      \item The function $\G$ is quasi-monotone.
      \item The function $t\mapsto \G(t x_1+(1-t)x_0)$ defined on $[0,1]$ is monotone  for all $x_0,x_1\in X$. 
   \end{enumerate}
\end{lemma}

\begin{proof}
  In the following, we fix some $x_0,x_1\in X$ and $t\in [0,1]$, and 
    define $x_t := t x_1+(1-t)x_0$.

  \atob i {ii} Without loss of generality we may assume $\G(x_0) \leq \G(x_1)$.
  Then quasi-monotonicity ensures $\G(x_0) \leq \G(x_t) \leq \G(x_1)$. Now let us fix an $s\in [0,t]$.
  Then $x_s$ is in the segment between $x_0$ and $x_t$ and hence we obtain by the same reasoning that 
  $\G(x_0) \leq \G(x_s) \leq \G(x_t)$.
  
  \atob {ii} {i} By assumption we have $\min\{\G(x_0), \G(x_1)  \} \leq \G(x_t) \leq \max\{\G(x_0), \G(x_1)  \}$,
  and this is equivalent to being both quasi-convex and quasi-concave.
\end{proof}

Our first result shows that for continuous functionals $\G:X\to \R$, quasi-monotonicity
is equivalent to the convexity of all level sets.

\begin{lemma}\label{quasi-m-char} 
Let $E$ be a topological vector space, $X\subset E$ be a convex subset and 
$\G:X\to \R$ be a continuous function. 
Then the following statements are equivalent:
\begin{enumerate}[itemsep=-.6ex]
 \item For all $r\in \im \G$, the level sets $\{\G=r\}$ are convex.
 \item For all $r\in \im \G$, the sets $\{\G<r\}$ and $\{\G>r\}$ are convex.
 \item The function $\G$ is quasi-monotone, i.e.~the sets $\{\G\leq r\}$ and $\{\G\geq r\}$ are convex
 for all  $r\in \im \G$.
\end{enumerate}
\end{lemma}

\begin{proof}
\atob i {ii} By symmetry, it suffices to consider
the case $\{\G<r\}$. Let us assume that  $\{\G<r\}$ is not convex. Then there exist
$x_0, x_1\in \{\G<r\}$ and an $\a\in (0,1)$ such that for
$x_\a := (1-\a)x_0 + \a x_1$ we have $x_\a \not\in \{\G<r\}$, that is $\G(x_\a)\geq r$.
Now, we first observe that, for $r_0:= \G(x_0)<r$ and $r_1:= \G(x_1)<r$,
we have $r_0 \neq r_1$, since $r_0=r_1$
would imply $\G(x_\a) \in \{\G = r_0\} \subset \{\G<r\}$ by the assumed convexity
of the level set $\{\G = r_0\}$.
Let us assume without loss of generality that $r_0 < r_1$. 
Then we have $r_1 \in (\G(x_0), \G(x_\a))$, and thus the 
intermediate value theorem applied to the continuous map 
$\b\mapsto \G( (1-\b)x_0 + \b x_\a)$ on $(0,1)$ yields a $\b^*\in (0,1)$ such that 
for $x^*:= (1-\b^*)x_0 + \b^* x_\a$ we have $\G(x^*) = r_1$.
Let us define $\g:= \frac{(1-\b^*)\a}{1-\b^*\a}$. Then we have 
$\g\in (0,1)$ and $x_\a = (1-\g) x^* + \g x_1$. By the assumed convexity of $\{\G = r_1\}$,
we thus conclude that $\G(x_\a) \in \{\G = r_1\} \subset \{\G<r\}$,
i.e.~we have found a contradiction.

\atob {ii} {iii} This follows from $\{\G\geq r\} = \bigcap_{r'<r} \{\G> r'\}$ and 
$\{\G\leq r\} = \bigcap_{r'>r} \{\G< r'\}$.

\atob {iii} i This follows from $\{\G=r\} = \{\G\leq r\} \cap \{\G\geq r\}$.
\end{proof}

\begin{lemma}\label{quasi-monotone-level-sets}
Let $E$ be a topological vector space, $X\subset E$ be a convex subset and 
$\G:X\to \R$ be a continuous, quasi-monotone function. Then the image $\im \G$ is an interval
and the sets $\{r<\G<s\}$ are convex, open, and non-empty 
for all $r,s\in \im \G$ with $r<s$.
\end{lemma}
 
\begin{proof}
Since $X$ is convex, it is connected, and thus $\G(X)$ is connected by the continuity of $\G$.
Since  the only connected sets in $\R$ are intervals, we conclude that  $\G(X)$ is an interval.

Moreover, the sets $\{r<\G<s\} = \{\G>r\} \cap \{\G< s\}$   are open by the continuity of $\G$, and 
Lemma \ref{quasi-m-char}  shows that they are also convex. To show that they are non-empty, we fix 
$r,s\in \im \G$ with $r<s$. Then we have $t:=(r+s)/2 \in \im \G$ since $\im \G$ is an interval, and thus 
there is an $x\in X$ with $\G(x) = t$. The construction now gives $x\in \{r<\G<s\}$.
\end{proof}

\begin{lemma}\label{quasi-monotone-extension}
   Let $E$ be a normed space, $X\subset E$ be a convex set and 
 $\G:X\to \R$ be a quasi-monotone function that has a continuous extension $\hat \G: \overline X\to \R$.
 Then  $\hat \G$ is quasi-monotone and we have 
 \begin{displaymath}
    \myint{\hat{\G} (\overline{X})} = \mathring{\G(X)}\, .
 \end{displaymath}

\end{lemma}
 
\begin{proof}
The quasi-monotonicity of $\hat \G$ can be easily established using \eqref{def-quasi-conv} and
the analogue inequality for quasi-concavity.
Moreover, since $\hat \G$ is an extension of $\G$, we obviously have $\G(X) \subset \hat \G(\overline X)$, and thus 
we find $\mathring{\G(X)} \subset \myint{\hat{\G} (\overline{X})}$. To show the converse inclusion,
we first note that the continuity of $\hat \G$ yields 
$\hat \G(\overline X) \subset \overline{\hat \G(X)} = \overline{\G(X)}$. 
Therefore, we find 
\begin{displaymath}
  \myint{\hat{\G} (\overline{X})} \subset \myint  \overline{\G(X)} =  \mathring{\G(X)}\, ,
\end{displaymath}
where in the last step we used that $\G(X)$ is an interval.
\end{proof}

\begin{lemma}\label{closure-of-levels-lemma-1}
 Let $E$ be a normed space, $X\subset E$ be a non-empty set and 
 $\G:X\to \R$ be a functional that has a continuous extension $\hat \G: \overline X\to \R$.
 Then, for all $r\in \R$, 
 we have 
 \begin{align}\label{closure-of-levels-lemma-1-h1}
  \{\hat \G>r\} &\subset \overline {\{\G> r   \}}^E\\  \label{closure-of-levels-lemma-1-h2}
   \overline {\{\G\geq  r   \}}^E & \subset \{\hat \G\geq r\} \, .
 \end{align}
\end{lemma}

\begin{proof}
 To prove  \eqref{closure-of-levels-lemma-1-h1} we  fix an
 $x\in  \{\hat \G>r\}$ and define $r^*:= \hat \G(x)$ and  $\e:= r^* -r$. Since $\e>0$, the continuity of 
 $\hat \G$ shows that $\hat \G^{-1}((r^*-\e,\infty))$ is open in $\overline X$ with $x\in \hat \G^{-1}((r^*-\e,\infty))$, and thus there exists a $\d>0$ such that 
 $(x+\d B_E) \cap \overline X \subset \hat \G^{-1}((r^*-\e,\infty))$. Moreover, $\{\hat \G>r\}\subset \overline X$ gives 
 a sequence $(x_n)\subset X$ such that $x_n\to x$. Clearly, we may assume without loss of generality that 
 $\enorm{x-x_n}\leq \d$ for all $n\geq 1$, and hence we 
 find $\G(x_n) = \hat \G(x_n) > r^*-\e = r$, for all $n\geq 1$, i.e.~$(x_n)\subset \{\G> r\}$.
 
 For the proof of \eqref{closure-of-levels-lemma-1-h2} we 
 first observe that ${\{\G\geq  r   \}} = {\{\hat \G\geq  r   \}} \cap X$, and thus we find
 \begin{displaymath}
    \overline{\{\G\geq  r   \}}^E 
    = \overline{\{\hat \G\geq  r   \} \cap X}^E 
    \subset \overline{\{\hat \G\geq  r   \}}^E \cap \overline X 
    =   {\{\hat \G\geq  r   \}} \cap \overline X 
    = {\{\hat \G\geq  r   \}} \, ,
 \end{displaymath}
where in the second to  last step we used 
 both the continuity of $\hat \G$ and the fact that $\overline X$ is closed.
\end{proof}

\begin{lemma}\label{closure-for-strict-mono}
   Let $E$ be a normed space, $X\subset E$ be a convex set and 
 $\G:X\to \R$ be a   functional that has a continuous and strictly quasi-monotone
 extension $\hat \G: \overline X\to \R$.
 Then, for all $r\in \mathring{\G(X)}$  
 we have 
 \begin{displaymath}
    \{\hat \G\geq r\} = \overline {\{\G\geq  r   \}}^E\, .
 \end{displaymath}
\end{lemma}

\begin{proof}
 {``$\supset$''.} This follows from inclusion  \eqref{closure-of-levels-lemma-1-h2} of Lemma \ref{closure-of-levels-lemma-1}.
%

 {``$\subset$''.} Let us fix an $x\in \{\hat \G\geq r\}$.  Since $\{\hat \G\geq r\} \subset  \overline X$, there then exists 
 a sequence $(x_n)\subset X$ with $x_n\to x$. Clearly, if $x_n\in \{\G\geq r\}$
 for infinitely many $n$, then there is nothing left to prove,
 and hence we assume that  $x_n\in \{\G< r\}$ for all $n\geq 1$. The continuity of $\hat \G$ then yields 
 $\G(x_n) \to \hat \G(x)$, and therefore we conclude that $\hat \G(x) \leq r$, that is $\hat \G(x) = r$.
 Let us now fix an $x^+\in \{\G>r\}$, which exists by Lemma  \ref{quasi-monotone-level-sets}
 and the fact that $\G$ is quasi-monotone.
 For $t\in [0,1]$ we further define  $x(t):= (1-t) x + tx^+$. Since $\overline X$ 
 inherits its convexity from $X$ and $x\in \overline X$, $x^+\in X$
 we then know that $x(t) \in \overline X$ for all $t\in [0,1]$.
 Moreover, the strict concavity of $\hat \G$ ensures $\hat\G(x(t)) > \min\{\G(x), \G(x^+)\} = r$ for all $t\in (0,1)$
 and thus we conclude by \eqref{closure-of-levels-lemma-1-h1} that 
 \begin{displaymath}
  x(t) \in \{\hat \G >r\} \subset  \overline {\{\G> r   \}}^E
 \end{displaymath}
 for all $t\in (0,1)$. For all $n\geq 2$ there thus exist an $x_n^+\in \{\G> r   \}$ 
 with $\enorm{x_n^+ - x(1/n)}\leq 1/n$.
 Since $\enorm{x(1/n) - x} \leq n^{-1}\enorm{x-x^+}$ we then obtain $x_n^+\to x$, which finishes the proof.
\end{proof}

\section{Proofs for Section \ref{sec:exist}}\label{sec:proof-1}


\begin{proof}[Proof of Lemma \ref{F-is-in-kernel}]
 Let us fix a $y\in F$. Since $F= \spann (-\xstar + B)$, there then exists $\a_1,\dots,\a_n\in \R$
and $x_1,\dots,x_n\in B$ such that $y = \sum_{i=1}^n \a_i(-\xstar+ x_i)$. By the linearity of $\p'$, this 
yields
 \begin{displaymath}
  \langle \p', y\rangle 
= \sum_{i=1}^n \a_i \bigl(\langle \p', x_i\rangle - \langle \p',\xstar\rangle\bigr)
= 0\, ,
 \end{displaymath}
where in the last step we used $\langle \p', x_i\rangle =1= \langle \p',\xstar\rangle$.
Now $\xstar\not\in F$ 
follows from the just established $F\subset \ker\p'$  and $\langle \p',\xstar\rangle= 1$.
In addition, we immediately obtain $F\cap \R\xstar = \{0\}$, and thus $F\oplus \R\xstar$ is indeed a direct sum.
Moreover, the equality  $F\oplus \R\xstar = \spann B$ follows from 
\begin{displaymath}
 \sum_{i=1}^n \a_i(-\xstar + x_i) + \a_0 \xstar = \sum_{i=1}^n \a_i x_i + \Bigl( \a_0 - \sum_{i=1}^n \a_i\Bigr)\xstar\, ,
\end{displaymath}
which holds for all $n\in \N$, $\a_0,\dots,\a_n\in\R$, and $x_1,\dots,x_n\in B$.
Now, $\hnorm\cdot $ can be 
constructed in the described way. Here we note, that the definition of $\hnorm\cdot$ resembles a standard way 
of defining norms on direct sums,  and thus $\hnorm\cdot$ is indeed a norm.
Furthermore,
$\enorm\cdot \leq \hnorm\cdot$ immediately follows from the construction of $\hnorm\cdot$ and 
the assumed $\enorm \cdot \leq \fnorm\cdot$. Finally, 
$\fnorm \cdot = \hnorm\cdot$ on $F$ is obvious and so is $B-B\subset F$.
\end{proof}


%
%
%
%

Our next little lemma 
%
shows that the space $H$ can also be generated from $F$ and an \emph{arbitrary} element of $B$.

\begin{lemma}\label{generating-H}
 If \textbf{B1} and \textbf{B2} are satisfied, then we have $F\oplus \R x_0 = H$ for all $x_0\in B$.
\end{lemma}

\begin{proof}
By $\p'(x_0)=1$ and the inclusion  $F\subset \ker \p'$ established in Lemma \ref{F-is-in-kernel}, 
we see that $x_0\not\in F$, and hence 
$F\cap \R x_0 = \{0\}$. 

The inclusion $F\oplus \R x_0 \subset H$ follows from the equality $H= \spann B$ established in Lemma \ref{F-is-in-kernel}
and
\begin{displaymath}
  \sum_{i=1}^n \a_i(-\xstar + x_i) + \a_0 x_0= \sum_{i=0}^n \a_i x_i - \sum_{i=1}^n \a_i\xstar\, ,
\end{displaymath}
which holds for all $n\in \N$, $\a_0,\dots,\a_n\in\R$, and $x_1,\dots,x_n\in B$.

To prove the converse inclusion, we first note that $-\xstar = (-\xstar + x_0) - x_0 \in F\oplus \R x_0$
implies $\R \xstar \subset F\oplus \R x_0$.  Since we also have $F\subset F\oplus \R x_0$,
we conclude by Lemma \ref{F-is-in-kernel} that  
$H = F\oplus \R \xstar \subset F\oplus \R x_0$.
\end{proof}

Our next lemma 
 shows that the cone decomposition \textbf{B3} makes it easier to decide whether a linear 
functional is continuous.

\begin{lemma}\label{continuity-E-test}
 Let 
 \textbf{B3} be satisfied. Then 
 a linear map $z':H\to \R$ is continuous with respect to $\enorm\cdot$, if and only if for all sequences 
$(z_n) \subset \cone B$ with $\enorm {z_n}\to 0$ we have $\langle z',z_n\rangle \to 0$.
\end{lemma}

\begin{proof}
``$\Rightarrow$ '': Since $\cone B\subset H$ by the definition of $H$, this implication is trivial. 

``$\Leftarrow$ '': By the linearity of $z'$ it suffices to show that $z'$ is $\enorm\cdot$-continuous at 0.
To show the latter, we fix a sequence $(z_n)\subset H$ with $\enorm {z_n}\to 0$. By \textbf{B3}
  there then exist sequences $(z_n^-), (z_n^+)\subset \cone B$ with $z_n = z_n^+ - z_n^-$
and $\enorm{z_n^-} + \enorm{z_n^+} \leq K\enorm{z_n}$. Consequently, we obtain $\enorm{z_n^-} \to 0$ and 
$\enorm{z_n^+}\to 0$, and thus our assumption together with the linearity of $z'$  yields 
$\langle z',z_n\rangle = \langle z',z_n^+\rangle - \langle z',z_n^-\rangle \to 0$
\end{proof}

In the following, we almost always need the assumption
 \textbf{B2} to be satisfied. In this case, we sometimes need to 
 consider two metrics on $B$, namely 
 the metric $d_E$ 
induced by  $\enorm\cdot$ and the 
metric $d_F$ induced by $\fnorm\cdot$ via translation, that is
\begin{align}\label{def-df}
d_F(x_1,x_2) := \fnorm{(-\xstar+x_1) - (-\xstar+x_2) } = \fnorm{x_1-x_2} = \hnorm{x_1-x_2}\, ,\qquad \quad x_1,x_2\in B,
\end{align}
where the last identity follows from Lemma \ref{F-is-in-kernel} provided that \textbf{B1} also holds.
Note that the  assumed   $\enorm \cdot \leq \fnorm\cdot$ immediately implies 
$d_E(x_1,x_2) \leq d_F(x_1,x_2)$ for all $x_1,x_2\in B$, and thus the identity map 
$\id:(B, d_F)\to (B, d_E)$ is   Lip\-schitz continuous.

The following result 
 collects some simple  properties of the sets 
$\{\G<r\}$ and $\{\G>r\}$ we wish to separate.

\begin{lemma}\label{convex-level-sets}
Let  \textbf{B2} and \textbf{G1} be satisfied. 
 Then  $\G(B)$ is an interval, and, for all $r\in \mathring \G(B)$, the sets
	$\{\G<r\}$ and $\{\G>r\}$ are non-empty, convex, and open in $B$ with respect to  $d_F$.
\end{lemma}
 
\begin{proof}
Clearly, the sets $\{\G<r\}$ and $\{\G>r\}$
are  open with respect to $d_F$, since $\G$ is assumed to be continuous with respect to 
$d_F$. 
The remaining assertions follow from 
the Lemma \ref{quasi-m-char} and \ref{quasi-monotone-level-sets}.
%
%
\end{proof}

Our next goal is to investigate relative interiors of subsets of $A$. We begin with a result 
that shows the richness of $\relint AF$.

\begin{lemma}\label{all-interirors}
Let   \textbf{B2*} and \textbf{G1} be satisfied. 
Then,  for all $r\in I$, there exists an $x\in \{\G=r\}$ such that $-\xstar+x\in\relint AF$.
\end{lemma}

\begin{proof}
 If $\xstar \in\{\G=r\}$ there is nothing to prove, and hence we may assume without loss
of generality that $\xstar \in\{\G>r\}$. Let us write $r^\star := \G(\xstar)$. Now, since 
$r\in I$ and $I$ is an open interval by Lemma \ref{convex-level-sets},
there exists an $s\in I$ with $s<r$. We fix an $x_0\in \{\G=s\}$ and, for $\lb\in [0,1]$,
we consider $x_\lb := \lb \xstar + (1-\lb)x_0$. Then we have $\G(x_0) = s < r<r^\star = \G(\xstar)$, and thus 
the intermediate theorem shows that there exists a $\lb\in (0,1)$ with $\G(x_\lb) = r$.
Our goal is to show that this $x_\lb$ satisfies $-\xstar+x_\lb\in\relint AF$. To this end, we recall that 
$0\in \relint AF$, which is ensured by \textbf{B2*},
gives an $\e>0$ such that for all $y\in F$ satisfying  $\fnorm{y}\leq \e$
we actually have $y\in A$. Let us write $\d:= \lb \e$. Then it suffices to show that, for all $y\in F$
satisfying $\fnorm {-\xstar + x_\lb - y}\leq \d$, we have $y\in A$. Consequently, let us fix such a $y\in F$.
For 
\begin{displaymath}
 \tilde x := \xstar + \frac {y-(1-\lb) (-\xstar + x_0)}{\lb}
\end{displaymath}
we then have $y = \lb(-\xstar + \tilde x) + (1-\lb) (-\xstar + x_0)$. By the convexity of $A$ and $-\xstar + x_0\in A$,
it thus suffices to 
show $-\xstar + \tilde x\in A$. However, the latter follows from 
\begin{align*}
 \fnorm{-\xstar + \tilde x} 
& = \lb^{-1} \fnorm{y - (1-\lb) (-\xstar + x_0)} \\
& = \lb^{-1} \fnorm{y - x_\lb + \xstar} \\
& \leq \lb^{-1} \d\, ,
\end{align*}
and thus the assertion is proven.
\end{proof}

Our last elementary result shows that having non-empty relative interior in $A$ implies 
a non-empty relative interior in $F$. This result will later be applied to translates of 
the open, non-empty sets $\{\G<r\}$ and $\{\G>r\}$.

\begin{lemma}\label{relative-interiors}
Let \textbf{B2*} be satisfied, and 
   $K\subset A$ be an arbitrary subset with $\relint KA\neq \emptyset$, that is $K$
 has  non-empty relative $\fnorm\cdot$-interior in $A$. 
Then, for all $y\in \relint KA$, there exists a 
$\d_y\in (0,1/2]$ such that $(1-\d)y \in \relint K F$ for all $\d\in (0,\d_y]$.
In particular, we have $\relint KF\neq \emptyset$.
\end{lemma}

\begin{proof}
By the assumed  $0\in \relint AF$, there exists an $\e_0\in (0,1]$ such that 
$\e_0B_F \subset A$. Moreover, the assumption $y\in \relint KA$ yields an $\e_1\in (0,\e_0]$ such that
\begin{equation}\label{relative-interiors-h1}
 (y+\e_1B_F) \cap A \subset K\, .
\end{equation}
We define $\d_y := \e_1/(\e_1 + \fnorm y)$. Then, it suffices to show that
\begin{equation}\label{relative-interiors-h2}
 (1-\d)y+\e_1\d B_F  \subset K
\end{equation}
 for all $\d\in(0,\d_y]$. To show the latter, we fix a $y_1\in \e_1\d B_F$.
An easy estimate then 
shows that $\fnorm{-\d y +y_1} \leq \d\fnorm y + \fnorm{y_1} \leq \d(\fnorm y + \e_1)\leq \e_1$,
and hence we obtain
\begin{displaymath}
 (1-\d)y + y_1 = y -\d y +y_1 \in (y + \e_1 B_F) \, .
\end{displaymath}
By \eqref{relative-interiors-h1} it thus suffices to show $(1-\d)y + y_1 \in A$.
Now, if $y_1=0$, then the latter immediately follows from 
$(1-\d)y + y_1 = (1-\d)y + \d \cdot 0$, the convexity of $A$, and $0\in A$.
Therefore, it remains to consider the case $y_1\neq 0$. Then we have
\begin{displaymath}
 \frac {\e_0}{\fnorm {y_1}} y_1 \in \e_0 B_F \subset A\, ,
\end{displaymath}
and $\frac {\fnorm {y_1}} {\e_0} \leq \frac{\e_1\d}{\e_0} \leq \d$. Consequently, the convexity of $A$ and $0\in A$ yield
\begin{displaymath}
 (1-\d)y + y_1 
= (1-\d) y 
+ \frac {\fnorm {y_1}} {\e_0} \biggl(\frac {\e_0}{\fnorm {y_1}} y_1\biggr) 
+ \biggl( \d - \frac {\fnorm {y_1}} {\e_0}\biggr) \cdot 0
\in A\, ,
\end{displaymath}
and hence \eqref{relative-interiors-h2} follows.
\end{proof}

Our next goal is to move towards the proof of Theorem \ref{unique-separation-char-1}.
This is done in a couple of intermediate results that successively establish 
more properties of certain, separating functionals.
We begin with a somewhat crude separation of convex subsets in $A$ that have an non-empty relative interior.

\begin{lemma}\label{weak-separation-in-A}
Let   \textbf{B2*} be satisfied, and 
   $K_-, K_+ \subset A$ be two convex sets with $\relint{K_\pm}A \neq \emptyset$ and 
$K_- \cap \relint{K_+}F = \emptyset$. Then there exist a $y'\in F'$ and an $s\in \R$
such that 
\begin{align*}
 K_- &\subset \{y'\leq s\} &\mbox{ and } \qquad \qquad \relint{K_-}F &\subset \{y'<s\}\, ,\\
K_+ &\subset \{y'\geq s\} & \mbox{ and }  \qquad\qquad \relint{K_+}F &\subset \{y'>s\}\, .
\end{align*}
Moreover, if $s\leq 0$, then we actually have $\relint{K_-}A \subset \{y'<s\}$,
and, if $s\geq 0$,  we  have $\relint{K_+}A \subset \{y'>s\}$.
\end{lemma}

\begin{proof}
 By Lemma \ref{relative-interiors} and the assumed $\relint{K_\pm}A \neq \emptyset$
we find $\relint{K_\pm}F \neq \emptyset$.
By a version of the Hahn-Banach separation theorem, see e.g.~\cite[Thm.~2.2.26]{Megginson98},
there thus exist a $y'\in F'$ and an $s\in \R$ such that 
\begin{align*}
 K_- &\subset \{y'\leq s\} \\
K_+ &\subset \{y'\geq s\} \\ 
\relint{K_+}F &\subset \{y'>s\}\, .
\end{align*}
Let us first show $\relint{K_-}F \subset \{y'<s\}$. To this end, we fix a $y_-\in \relint {K_-}F$
and a $y_+\in \relint{K_+} F$. Since $\relint{K_-} F$ is open in $F$, there then exists 
a $\lb\in (0,1)$ such that 
 \begin{displaymath}
  \lb y_+ + (1-\lb) y_- = y_- + \lb(y_+-y_-) \in \relint{K_-} F\subset K_-\, .
 \end{displaymath}
From the latter and the already obtained inclusions we conclude that
\begin{align*}
 s 
 \geq \bigl\langle y' ,  \lb y_+ + (1-\lb) y_- \bigr\rangle
 = \lb \langle y', y_+\rangle + (1-\lb) \langle y', y_-\rangle
 > \lb s + (1-\lb) \langle y', y_-\rangle\, .
\end{align*}
Now, some simple transformations together with $\lb\in (0,1)$ yield $\langle y', y_-\rangle < s$, i.e.~we
have shown $\relint{K_-}F \subset \{y'<s\}$.

Let us now show that $s\leq 0$ implies $\relint{K_-}A \subset \{y'<s\}$.
To this end, we assume
 that  there exists a $y\in \relint{K_-}A$ with $\langle y', y\rangle \geq s$.
Since $\relint{K_-}A\subset K_-$, the already established inclusion 
$K_- \subset \{y'\leq s\}$ then yields $\langle y', y\rangle = s$.
Moreover, by Lemma \ref{relative-interiors} there exists a $\d>0$ such that 
$(1-\d)y\in \relint {K_-}F$. From the previously established $\relint{K_-}F \subset \{y'<s\}$
we thus obtain
\begin{displaymath}
 s > \bigl\langle y', (1-\d) y\bigr\rangle = (1-\d) s\, .
\end{displaymath}
Clearly, this yields $\d s>0$, and since $\d>0$, we find $s>0$.
The remaining implication can be shown analogously.
\end{proof}

The next result refines the separation of Lemma \ref{weak-separation-in-A} under additional assumptions 
on the sets that are to be separated. 
Its assertion, but not its proof, mimics the first part of Step 2 of the proof of Theorem 5 of \cite{LambertXXa}.

\begin{proposition}\label{strong-separation-in-A}
Let  \textbf{B2*} be satisfied, and 
   $K_-, K_0, K_+ \subset A$ be mutually disjoint, non-empty convex sets with $\relint{K_\pm}A =K_\pm$
and $A= K_-\cup K_0\cup K_+$. Furthermore, assume that, for all $y\in K_0$ and $\e>0$,
we have $K_- \cap (y + \e B_F)\neq \emptyset$ and $K_+ \cap (y + \e B_F)\neq \emptyset$.
Then there exist a $y'\in F'$ and an $s\in \R$ such that 
\begin{align*}
 K_- &= \{y'< s\} \cap A \\
K_0 &= \{y'= s\} \cap A\\ 
K_+ &= \{y'>s\} \cap A\, .
\end{align*}
\end{proposition}

\begin{proof}
	We first observe that we clearly have 
$\relint{K_\pm}A = K_\pm \neq \emptyset$ and 
$K_- \cap \relint {K_+}F \subset K_-\cap K_+ = \emptyset$. 
Consequently, Lemma \ref{weak-separation-in-A}
provides a $y'\in F'$ and an $s\in \R$ that
satisfy the inclusions listed in Lemma \ref{weak-separation-in-A}.

Our first goal is to show $K_0 = \{y'= s\} \cap A$. 
To prove  $K_0 \subset  \{y'= s\} \cap A$, we fix a $y\in K_0$.  
Since $K_- \cap (y + \e B_F)\neq \emptyset$ for all $\e>0$, we then find a sequence $(y_n)\subset K_-$
such that $y_n\to y$. By Lemma \ref{weak-separation-in-A} we then
obtain
\begin{displaymath}
 \langle y', y\rangle = \lim_{n\to \infty} \langle y',y_n\rangle \leq s\, ,
\end{displaymath}
i.e.~$y\in \{y'\leq s\} \cap A$. 
Using $K_+ \cap (y + \e B_F)\neq \emptyset$ for all $\e>0$,
we can analogously show $y\in \{y'\geq s\} \cap A$,
and hence we obtain $y\in \{y'= s\} \cap A$. 

To show the inclusion $\{y'= s\} \cap A \subset K_0$, we assume without loss of generality that $s\geq 0$.
Let us now fix a $y\in A\setminus K_0$, so that our goal becomes to show $y\not\in \{y'= s\} \cap A$.
Now, if $y\in K_+$, we obtain $\langle y', y\rangle > s$, since we have seen in Lemma 
\ref{weak-separation-in-A} that $s\geq 0$ implies 
$K_+ = \relint{K_+}A \subset \{y'>s\}$. Therefore, it remains to consider the case $y\in K_-$.
Let us fix a $y_1\in K_+$. Then we have just seen that $\langle y',y_1\rangle >s$.
For $\lb\in [0,1]$ we now define $y_\lb := \lb y_1 + (1-\lb) y$.
Now, if there is a $\lb\in (0,1)$ with $\langle y', y_\lb\rangle = s$, we obtain 
\begin{align*}
 s = \bigl\langle y',  \lb y_1 + (1-\lb) y\bigr\rangle 
= \lb \langle y', y_1\rangle + (1-\lb) \langle y', y\rangle 
 > \lb s + (1-\lb) \langle y', y\rangle \, ,
\end{align*}
that is $\langle y', y\rangle < s$. Consequently, it remains to show the existence of such a $\lb\in (0,1)$.
Let us assume the converse, that is $y_\lb\in K_-\cup K_+$ for all $\lb\in (0,1)$
by the already established $K_0 \subset  \{y'= s\} \cap A$.
Since $y_0 = y \in K_-$ and $y_1 \in K_+$, we then
have
\begin{equation}\label{strong-separation-in-A-h1}
y_\lb\in K_-\cup K_+  
\end{equation}
for \emph{all} $\lb\in [0,1]$. Let us now consider the map $\psi:[0,1]\to A$ defined by $\psi(\lb) :=y_\lb$.
Clearly, $\psi$ is continuous, and since $K_\pm = \relint{K_\pm}A$, the 
pre-images $\psi^{-1}(K_-)$ and $\psi^{-1}(K_+)$ are open, and, of course,
disjoint. Moreover,  by $\psi(0) = y_0= y\in K_-$
and $\psi(1) = y_1 \in K_+$, they are also non-empty, and \eqref{strong-separation-in-A-h1}
ensures $\psi^{-1}(K_-) \cup \psi^{-1}(K_+) = [0,1]$. Consequently, we have 
found a partition of $[0,1]$ consisting of two open, non-empty sets, i.e.~$[0,1]$ is not connected.
Since this is obviously false, we found a contradiction finishing the proof of 
$\{y'= s\} \cap A \subset K_0$.

To prove the remaining two equalities, let us again assume without loss of generality that $s\geq 0$.
By Lemma \ref{weak-separation-in-A}, we then know 
$K_+ = \relint{K_+}A \subset \{y'>s\}\cap A$. Conversely, for $y\in \{y'>s\}\cap A$ we have 
already shown $y\not \in K_0$, and by the inclusion $K_-\subset \{y'\leq s\}$ established in 
Lemma \ref{weak-separation-in-A} we also know $y\not \in K_-$. Since 
$A= K_-\cup K_0\cup K_+$, we conclude that $y\in K_+$. Consequently, we have also shown 
$K_+ = \{y'>s\} \cap A$, and the remaining $K_- = \{y'<s\} \cap A$ now immediately follows.
\end{proof}

The next result, whose assertion 
mimics the second part of Step 2 as well as  Step 3 of the proof of Theorem 5 in an earlier version of \cite{LambertXXa},
 shows the existence of the separating families considered in Theorem \ref{unique-separation-char-1} and 
 Theorem \ref{unique-separation-char-2}.
The construction idea \eqref{def-zp}
of $z'$ and the 
proof of its $\enorm\cdot$-continuity is an abstraction
from Lambert's proof. However, the remaining parts of our proof
 heavily rely on the preceding results of this section and are therefore independent of \cite{LambertXXa}.

\begin{theorem}\label{level-set-identification-in-H}
Let \textbf{B1}, \textbf{B2*},  \textbf{G1}, and \textbf{G2} be satisfied. Then, 
 for all $r\in I$, there exists a $z'\in H'$ such that 
\begin{align*}
 \{\G<r\} &= \{z'< 0\} \cap B \\
\{\G=r\} &= \{z'= 0\} \cap B\\ 
\{\G>r\} &= \{z'>0\} \cap B\, .
\end{align*}
If, in addition, \textbf{B3} and \textbf{G1*} are satisfied, then $z'$ is actually continuous with respect to $\enorm\cdot$.
\end{theorem}

\begin{proof}
 For some fixed $r\in I$ we  consider the sets 
\begin{align*}
 K_-&:= -\xstar +  \{\G<r\} \\
K_0&:= -\xstar +  \{\G=r\}\\
K_+&:= -\xstar +  \{\G>r\}\, .
\end{align*}
Our first goal is to show that these sets satisfy the assumptions of Proposition \ref{strong-separation-in-A}.
To this end, we first observe that $\{\G<r\}\subset B$ immediately implies $K_- \subset -\xstar +B =A$, and 
the same argument can be applied to $K_0$ and $K_+$. Moreover, they are mutually disjoint since the defining
level sets are mutually disjoint, and since $r\in \mathring \G(B)$ they are also non-empty. The equality $A= K_-\cup K_0\cup K_+$
follows from $B=  \{\G<r\} \cup  \{\G=r\}\cup  \{\G>r\}$, and the convexity of $K_-$ and
$K_+$ is a consequence of the convexity of $\{\G<r\}$ and $\{\G>r\}$ established in Lemma \ref{convex-level-sets}.    
The convexity of $K_0$ follows from \textbf{G1}.
%
Moreover, by Lemma \ref{convex-level-sets}, the set $\{\G<r\}$ is open in $B$ with respect to $d_F$,
and since the metric spaces $(B,d_F)$ and $(A, \fnorm\cdot)$ are isometrically isomorphic via translation with $-\xstar$,
we see that $K_-$ is open in $A$ with respect to $\fnorm\cdot$. This shows $\relint{K_-}A= K_-$, and 
$\relint{K_+}A= K_+$ can be shown analogously. Finally, observe that for $x\in \{\G=r\}$, $\e>0$, 
 and $y:= -\xstar + x$ we have
\begin{align*}
 K_-\cap (y+\e B_F) 
&= \bigl(-\xstar +  \{\G<r\}\bigr) \cap \bigl(-\xstar + x + \e B_F\bigr)\\
&= \bigl(-\xstar +  \{\G<r\}\bigr) \cap \bigl(-\xstar + x + \e B_H\bigr)\\
& = -\xstar + \bigl( \{\G<r\} \cap (x + \e B_H) \bigr) \\
& \neq \emptyset\, ,
\end{align*}
where 
in the second step we used the fact $\fnorm \cdot = \hnorm\cdot$ on $A\subset F$, see Lemma \ref{F-is-in-kernel},
and 
the last step relies on   \textbf{G2}.
Obviously, $ K_+\cap (y+\e B_F) \neq \emptyset$ can be shown analogously, and hence, the assumptions of 
Proposition \ref{strong-separation-in-A} are indeed satisfied.  

Now, let $y'\in F'$ and $s\in \R$  be according to 
Proposition \ref{strong-separation-in-A}. 
Moreover, let $\hat y'\in H'$ be the extension of $y'$ to $H$ that is defined by
\begin{displaymath}
 \langle \hat y', y+\a\xstar\rangle := \langle y',y\rangle
\end{displaymath}
for all $y+\a\xstar \in H=F\oplus\R\xstar$.
Clearly, $\hat y'$ is indeed an extension of $y'$ to $H$ and the continuity of $\hat y'$ on $H$
follows from 
\begin{displaymath}
|\langle \hat y', y+\a\xstar\rangle| 
= |\langle y',y\rangle| 
\leq \snorm{y'}\cdot\fnorm y 
\leq \snorm{y'}\cdot \hnorm{y+\a\xstar}\, .
\end{displaymath}
With these preparations, we now define a $z'\in H'$ by
\begin{align}\label{def-zp}
 \langle z', z\rangle := -s \langle \p',z\rangle + \bigl\langle \hat y', z-\langle \p',z\rangle \xstar\bigr\rangle\, , \qquad \qquad z\in H.
\end{align}
Obviously, $z'$ is  linear. Moreover, the restriction $\p'_{|H}$ of $\p'$ to $H$ is continuous
with respect to $\hnorm\cdot$, since Lemma \ref{F-is-in-kernel}
ensured $\enorm\cdot \leq \hnorm\cdot$ on $H$, and consequently we obtain  $z'\in H'$. 

Let us show that $z'$ is the desired functional. To this end, we first observe that
the inclusion $F\subset \ker \p'$ established in Lemma \ref{F-is-in-kernel} together with $\xstar \in B\subset \{\p'=1\}$
yields $\xstar + F\subset \{\p' = 1\}$.
For $x \in \xstar + F\subset H$ this gives
\begin{align*}
 \langle z', x\rangle 
 = -s\langle \p',x\rangle   + \bigl\langle \hat y', x-\langle \p',x\rangle \xstar\bigr\rangle 
 = -s + \bigl\langle \hat y', x-  \xstar\bigr\rangle 
 = -s + \bigl\langle y', x-  \xstar\bigr\rangle \, .
\end{align*}
Moreover, recall that we have $x\in B$ if and only if $-\xstar +x\in A$, and hence we obtain
\begin{align*}
 \{z'= 0\} \cap B 
& = \{x\in B: \langle y', x-\xstar\rangle = s \} \\
& = \bigl\{x\in B: -\xstar+ x \in \{y'=s\} \bigr\} \\
& = \xstar + \bigl\{y\in A:  y \in \{y'=s\} \bigr\} \\
& = \xstar + \bigl( \{y'=s\} \cap A\bigr)\\
& = \xstar + K_0\\
& = \{\G=r\}\, .
\end{align*}
The remaining equalities $\{\G<r\} = \{z'< 0\} \cap B$ and  $\{\G>r\} = \{z'> 0\} \cap B$ can be shown analogously.

Let us finally show that the  functional $z'$ found so far is actually continuous with respect to $\enorm\cdot$,
if \textbf{B3} and \textbf{G1*} are satisfied.
Let us assume the converse. By Lemma \ref{continuity-E-test}, there then exists a sequence $(z_n)\subset \cone B$
with $\enorm{z_n}\to 0$ and $\langle z', z_n\rangle \not\to 0$. Picking a suitable subsequence and scaling it appropriately,
we may assume without loss of generality that either $\langle z', z_n\rangle < -1$ for all $n\geq 1$,
or $\langle z', z_n\rangle > 1$ for all $n\geq 1$.
Let us consider the first case, only, the second case can be treated analogously.
We begin by picking an $x_0\in \{\G>r\} = \{z'>0\} \cap B$. This yields $\a:= \langle z',x_0\rangle >0$. 
Moreover, since $(z_n)\subset \cone B$ and $z_n\neq 0$ by the assumed 
$\langle z', z_n\rangle < -1$, we find sequences $(\a_n)\subset (0,\infty)$ and $(x_n)\subset B$ such that 
$z_n = \a_nx_n$ for all $n\geq 1$. 
Our first goal is to show that $\a_n\to 0$. To this end, we observe that $x_n\in B\subset \{\p'=1\}$ implies
$1=|\langle \p',x_n\rangle | \leq \snorm{\p'}\cdot \enorm{x_n}$, and hence we obtain
\begin{displaymath}
 |\a_n| \leq |\a_n| \cdot \snorm{\p'}\cdot \enorm{x_n} = \snorm{\p'}\cdot \enorm{z_n} \to 0\, .
\end{displaymath}
For $n\geq 1$, we define $\b_n := \frac 1 {1+\a\a_n}$. Our considerations made so far then yield
both $\b_n\to 1$ and 
 $\b_n\in (0,1)$
for all $n\geq 1$. By the definition of $\a$ and the assumptions made on $(z_n)$, this yields
\begin{equation}\label{level-set-identification-in-H-h1}
 \langle z', \b_n(x_0 + \a z_n)\rangle  
=  \b_n \bigl( \a + \a\langle z', z_n\rangle \bigr)
<  0
\end{equation}
for all $n\geq 1$. On the other hand,  $x_0\in \{\G>r\}$ ensures $\frac {\G(x_0)-r}2>0$, and 
since \textbf{G1*} assumes that $\G$ is 
$\enorm\cdot$-continuous,
there thus exists a $\d>0$ such that, for all 
$x\in B$ with $\enorm{x-x_0} \leq \d$, we have 
\begin{displaymath}
 \bigl|\G(x) - \G(x_0)\bigr| \leq \frac {\G(x_0)-r}2\, .
\end{displaymath}
For such $x$, a simple transformation then yields $\G(x) \geq \frac{\G(x_0)+r}2 > r$, and thus we find
\begin{displaymath}
 \bigl\{ x\in B: \enorm{x-x_0} \leq \d\bigr\} \subset \{\G>r\} = \{z'>0\} \cap B\, .
\end{displaymath}
To find a contradiction to \eqref{level-set-identification-in-H-h1}, it thus suffices to show that 
\begin{equation}\label{level-set-identification-in-H-h2}
\b_n(x_0 + \a z_n) \in \{ x\in B: \enorm{x-x_0} \leq \d\}
\end{equation} for all sufficiently large $n$. To prove this, we first observe that 
\begin{displaymath}
 \b_n(x_0 + \a z_n) = \b_n x_0 + \frac{\a\a_n}{1+\a\a_n} x_n = \b_n x_0 + (1-\b_n)x_n\, ,
\end{displaymath}
and since $\b_n\in (0,1)$, the convexity of $B$ yields $\b_n(x_0 + \a z_n)\in B$.
Finally, we have 
\begin{displaymath}
 \enorm{x_0 - \b_n(x_0 + \a z_n)} \leq (1-\b_n)\enorm{x_0} + \a\b_n \enorm{z_n} \to 0
\end{displaymath}
since $\b_n\to 1$ and $\enorm{z_n}\to 0$. Consequently, \eqref{level-set-identification-in-H-h2}
is indeed satisfied for all sufficiently large $n$, which finishes the proof.
\end{proof}

Theorem \ref{level-set-identification-in-H} has shown the existence of a functional 
separating the level sets of $\G$. Our next and final  goal is to show that this functional
is unique modulo normalization.
To this end, we need the following lemma, which shows that the null space of a separating functional
is completely determined by the set  $\{\G=r\}$.

Note that the assertion of the
Lemmas \ref{kernel-is-determined} and  \ref{equal-modulo-orientation}
are  inspired by Step 3 of the proof of Theorem 5 of \cite{LambertXXa}, but again our proofs are   more
complicated, since we cannot guarantee $\xstar\in \{\G=r\}$.

\begin{lemma}\label{kernel-is-determined}
Let \textbf{B1}, \textbf{B2*},  and \textbf{G1} be satisfied. Moreover,   
 let $r\in I$ and $z:H\to \R$ be a linear functional satisfying $\{\G=r\} = B \cap \ker z'$. Then we have  $z'\neq 0$ and
 \begin{displaymath}
    \ker z' = \spann (\ker z'\cap B) = \spann \{\G=r\}\, .
 \end{displaymath}
%
%
\end{lemma}

\begin{proof}
The second equality is obvious, and 
 since $\ker z'$ is a subspace, the inclusion $\spann (\ker z'\cap B) \subset \ker z'$ is also  obvious.

To prove the converse inclusion, we fix a $z\in \ker z'$. Moreover, using Lemma \ref{all-interirors}, we fix 
an $x_0\in \{\G=r\} = B \cap \ker z'$ satisfying $-\xstar + x_0 \in \relint AF$.
By $z\in \ker z'\subset H$ and Lemma  \ref{generating-H}, which showed $H= F\oplus \R x_0$, there then  
exist a $y\in F$ and an $\a\in \R$ such that $z = y + \a x_0$. 
Obviously, it suffices to show both $\a x_0 \in \spann (\ker z'\cap B)$ and 
$y \in \spann (\ker z'\cap B)$.
Now,  $\a x_0 \in \spann (\ker z'\cap B)$ immediately follows from $x_0\in  \ker z'\cap B$,
and for $y=0$ the second inclusion is trivial. 
Therefore, let us assume that $y\neq 0$. 
Since  $-\xstar + x_0 \in \relint AF$, there then exists an $\e>0$ such that for all
$ \breve y \in F$ with $\fnorm{-\xstar + x_0 - \breve y}\leq \e$ we have $\breve   y\in A$.
Writing  $\hat y := \frac{\e}{\fnorm y} y$, we have $\hat y\in F$ by the assumed $y\in F$,  
and thus also $\tilde y:= -\xstar + x_0 + \hat y \in F$.
Moreover, our construction immediately yields $\fnorm{-\xstar + x_0 - \tilde y} = \e$, and hence we actually have 
$\tilde y\in A = -\xstar + B$. Consequently, we have found $x_0 + \hat y = \tilde y + \xstar\in B$.
On the other hand,  the assumed 
$x_0\in \ker z'$ implies $\a x_0\in \ker z'$, and  thus we find $y\in \ker z'$ by $z\in \ker z'$ and $z = y + \a x_0$. 
Using both $x_0,y\in \ker z'$, we thus 
obtain $x_0 + \hat y\in \ker z'$, which together with the already established $x_0 + \hat y \in B$
 shows $x_0+\hat y\in \spann  (\ker z'\cap B)$. Since   $x_0\in B \cap \ker z'$ by assumption
we therefore finally find the desired $y \in \spann  (\ker z'\cap B)$ by the definition of $\hat y$.

Finally, assume that $z' = 0$. By Lemma \ref{quasi-monotone-level-sets} in combination with 
\textbf{G1} and Lemma \ref{quasi-m-char} 
we find  an $x\in \{\Gamma <r\}$, and  the assumed $z'=0$
implies $x\in \ker z'$, while $\{\Gamma <r\}\subset B$ implies $x\in B$. This yields 
$x\in B \cap \ker z' = \{\G=r\}$, which contradicts  $x\in \{\Gamma <r\}$.
\end{proof}

The following lemma shows that, modulo orientation, two normalized separating functionals are equal.

\begin{lemma}\label{equal-modulo-orientation}
Let \textbf{B1},  \textbf{B2*}, and \textbf{G1}
be satisfied. Moreover, 
 let $r\in I$ and $z_1', z_2'\in H'$ such that 
  $\{\G=r\} = B \cap \ker z'_1$ and $\{\G=r\} \subset B \cap \ker z'_2$.
 Then there exists an $\a\in \R$ such that $z_2' = \a z_1'$, and if 
 $\{\G=r\} = B \cap \ker z'_2$, we actually have $\a\neq 0$.
\end{lemma}

\begin{proof}
 Our assumptions guarantee $B \cap \ker z'_1 \subset B \cap \ker z'_2 \subset \ker z_2'$, 
 and thus Lemma \ref{kernel-is-determined} yields
$\ker z_1' \subset \ker z_2'$.  Moreover,  Lemma \ref{kernel-is-determined} shows $z_1'\neq 0$, which in turn 
gives  a $z_0\in H$
with $z_0\not\in   \ker z_1'$. For $z\in H$, an easy calculation then shows that 
\begin{displaymath}
 z - \frac{\langle z_1', z\rangle}{\langle z_1', z_0\rangle}z_0 \in \ker z_1' \subset \ker z_2'\, ,
\end{displaymath}
and hence we conclude that $\langle z_2', z\rangle =  \frac{\langle z_1', z\rangle}{\langle z_1', z_0\rangle}\langle z_2',z_0\rangle$.
In other words, for $\a:= \frac{\langle z_2', z_0\rangle}{\langle z_1', z_0\rangle}$, we have 
$z_2' = \a z_1'$. Finally, $\{\G=r\} = B \cap \ker z'_2$ implies $z_2'\neq 0$ by Lemma \ref{kernel-is-determined}, and 
hence we conclude that $\a\neq 0$.
%
%
\end{proof}

\begin{proof}[Proof of Theorem \ref{unique-separation-char-1}]
  \atob i {iv} The existence has been proven in the first part of  Theorem \ref{level-set-identification-in-H}.
  To show the uniqueness, 
we assume that we have two normalized separating families $(z_r')_{r\in I} \subset H'$
   and $(\tilde z_r')_{r\in I} \subset H'$ for $\G$. Moreover, we fix an $r\in I$.
    Then Lemma \ref{equal-modulo-orientation} gives an $\a\neq 0$ with 
 $z_r' = \a \tilde z_r'$.  The imposed normalization $\snorm{z_r'}_{H'} = 1 = \snorm{\tilde z_r'}_{H'}$
 implies $|\a|=1$, and the orientation of $z_r'$ and $\tilde z_r'$ on $\{\G<r\}$ excludes the case $\a=-1$.
 Thus we have $z_r' =\tilde z_r'$.
   
   \atob {iv} {iii} Trivial.
   
   \atob {iii} {ii} By \eqref{unique-separation-a1} and \eqref{def-df} 
   we know that $\{\G<r\}$ is relatively open in $B$ with respect to $d_F$ for all $r\in I$.
   Moreover, for $r<\inf I$ we have $\{\G<r\}=\emptyset$ and for $r>\sup I$ we have $\{\G<r\}=B$.
   Finally, if $r:= \sup I< \infty$, then 
   \begin{displaymath}
   \{\G<r\} = \bigcup_{n\geq 1}  \{\G<r - 1/n\} \, ,
   \end{displaymath}
  and therefore $\{\G<r\}$ is relatively open in $B$ with respect to $d_F$ for \emph{all} $r\in\R$.
   Consequently, $\G$ is upper semi-continuous with respect to $d_F$, and analogously we can show that 
   $\G$ is lower semi-continuous with respect to $d_F$. Together, this gives the $\fnorm\cdot$-continuity of $\G$.
   Moreover, \eqref{unique-separation-a2} together with the convexity of $B$ shows that 
   $\{\G=r\}$ is convex for all $r\in I$, and by Lemma \ref{quasi-m-char} we conclude that $\G$ is quasi-monotone.
   To verify that $\G_{|B_0}$ is  strictly quasi-monotone, we fix $x_0,x_1\in B_0$ and write 
   $x_t := (1-t)x_0 + tx_1$ for $t\in [0,1]$. Furthermore, we define $r_0:= \G(x_0)$ and $r_1:= \G(x_1)$ and assume
   without loss of generality that $r_0\leq r_1$. To check that $\G_{|B_0}$ is quasi-monotone
   we first observe that  the already established 
    quasi-monotonicity of $\G$ yields $r_0 \leq \G(x_t) \leq r_1$ for all $r\in [0,1]$.
    Moreover, we have  $r_0,r_1\in I$, and 
    since $I$ is an  interval by Lemma \ref{quasi-monotone-level-sets}, we thus find $\G(x_t) \in I$. In other words, we have shown
    that $x_t\in B_0$ for all $t\in [0,1]$, and since the latter gives $\G_{|B_0}(x_t) = \G(x_t)$, we obtain the 
    quasi-monotonicity of $\G_{|B_0}$. Let us finally show that $\G_{|B_0}$ is strictly quasi-monotone.
    To this end, we keep our notation and additionally assume that $r_0<r_1$.
    Then, an easy calculation using \eqref{unique-separation-a2}, \eqref{unique-separation-a3}, and
    $x_1\in \{\G=r_1\}  \subset\{\G>r_0\}$ shows 
    \begin{displaymath}
       \langle z_{r_0}',x_t\rangle 
       = (1-t) \langle z_{r_0}',x_0\rangle  + t\langle z_{r_0}',x_1\rangle 
       =t \langle z_{r_0}',x_1\rangle 
       > 0
    \end{displaymath}
      for $t\in (0,1)$ and thus $x_t\in \{z_{r_0}'>0\} \cap B = \{\G>r_0\}$, that is $\G_{|B_0}(x_t) > r_0$.
      By considering $z_{r_1}'$ instead, we analogously obtain $\G_{|B_0}(x_t) < r_1$, and hence $\G_{|B_0}$ is indeed
      strictly quasi-monotone.

   \atob {ii} i Assumption \textbf{G1} follows from Lemma \ref{quasi-m-char}. To show that \textbf{G2} is also satisfied, we 
   fix an $r\in I$ and an $x\in \{\G=r\}$. By Lemma \ref{quasi-monotone-level-sets}
   there then exist an $s\in I$ with $s>r$ and an $x_1^+\in \{r<\G<s\}$. 
     For $t\in (0,1)$ we define $x_t^+ := (1-t) x + t x_1^+$. Then our construction ensures $x,x_1^+\in B_0$ and hence 
     the strict quasi-concavity of $\G_{|B_0}$ gives $\G(x_t^+) = \G_{|B_0} (x_t^+) > \min \{\G(x), \G(x_1^+)   \} = r$,
     that is $x_t^+\in \{\G>r\}$. Analogously we find $x_t^-\in \{\G<r\}$, and by choosing sufficiently small $t$ we can verify 
     \textbf{G2}.
\end{proof}

\begin{proof}[Proof of Theorem \ref{unique-separation-char-2}]
   \atob i {iv} The existence follows from the second part of  Theorem \ref{level-set-identification-in-H}.
    Since every $z'\in (H, \enorm\cdot)'$ is also an element of $H'$, the uniqueness 
   can be shown as in the proof of Theorem \ref{unique-separation-char-1}.
   
   \atob {iv} {iii} Trivial.
   
   
   \atob {iii} {ii} The $\enorm\cdot$-continuity can be  
   shown 
   as in the proof of Theorem \ref{unique-separation-char-1} if $d_F$ and $\fnorm\cdot$ are
   replaced by $d_E$ and $\enorm\cdot$, respectively. The remaining parts follow from Theorem \ref{unique-separation-char-1}
   and the already mentioned inclusion $ (H, \enorm\cdot)'\subset H'$.
   
   \atob {ii} i  This again follows from Theorem \ref{unique-separation-char-1}.
   
   Finally,  if \textbf{B4} is also satisfied, i.e.~if  $H$ is dense in $E$, 
then the existence of the unique
extension follows from  e.g.~\cite[Theorem 1.9.1]{Megginson98}. Moreover, this theorem also shows that 
$\snorm{\hat z_r'}_{E'} = \snorm{z_r'}_{E'}=1$. 
\end{proof}

%

%

Before we can prove Theorem \ref{unique-extended-separation} we need to establish a simple auxiliary result.

\begin{lemma}\label{closure-of-levels-lemma-2}
 Assume that \textbf{G1*} and \textbf{G3} are satisfied. Then for all  $r\in I$ we have 
  \begin{displaymath}
  \{\hat \G \geq r\} = \overline {\{\G \geq r  \}}^E\, .
 \end{displaymath}
\end{lemma}

\begin{proof}
 {``$\supset$''.} This follows from inclusion 
 \eqref{closure-of-levels-lemma-1-h2} of Lemma \ref{closure-of-levels-lemma-1}.

 {``$\subset$''.} Let us fix an $x\in   {\{\hat \G \geq r  \}}$. We write $r^*:= \G(x)$. By \textbf{G3} there then 
 exist $x_n\in \{\hat \G>r^*\}$ with $\enorm{x-x_n}\leq 1/n$ for all $n\geq 1$. Now, $r^*\geq r$ together with Lemma \ref{closure-of-levels-lemma-1}
 yields
 \begin{displaymath}
  x_n\in \{\hat \G>r^*\} \subset \{\hat \G>r\} \subset \overline {\{\G> r   \}}^E
 \end{displaymath}
 for all $n\geq 1$,
 and thus we find 
  $x\in  \overline {\{\G> r   \}}^E$.
\end{proof}

\begin{proof}[Proof of Theorem \ref{unique-extended-separation}]
 Using Lemma \ref{closure-of-levels-lemma-2} and Theorem \ref{unique-separation-char-2} we find
 \begin{align}\label{unique-extended-separation-p1}
  \{\hat \G \geq r\} 
  = \overline {\{\G \geq r  \}}^E 
  = \overline{ \{ z_r' \geq 0\} \cap B }^E\, .
%
 \end{align}
 Let us define $\Upsilon:B\to \R$ by $\Upsilon(x) :=  z_r'(x)$ for $x\in B$. Then 
 $\hat \Upsilon :=(\hat z_r')_{|\overline B}$ is clearly a continuous and strictly quasi-monotone extension 
 of $\Upsilon$ to $\overline B$. Moreover, Theorem \ref{unique-separation-char-2} in combination with 
 Lemma \ref{quasi-monotone-level-sets}
 shows 
 $\{\Upsilon >0\} = \{\G>r\}\neq \emptyset$
 and $\{\Upsilon <0\} = \{\G<r\}\neq \emptyset$, and using that $\Upsilon(B)$ is an interval
 by Lemma \ref{quasi-monotone-level-sets}
 we conclude that 
 $0\in \mathring {\Upsilon(B)}$. Consequently, Lemma \ref{closure-for-strict-mono}
 yields
  \begin{displaymath}
   \{\hat z_r'\geq 0  \} \cap \overline B 
   =\{\hat \Upsilon\geq 0\}
   = \overline {\{\Upsilon\geq  0   \}}^E
   = \overline{\{z_r'\geq 0\} \cap B   }^E\, .
 \end{displaymath}
 Combining this equality with  \eqref{unique-extended-separation-p1} we then
 find $ \{\hat \G \geq r\}  = \{\hat z_r'\geq 0  \} \cap \overline B$.
 Analogously, we can prove $ \{\hat \G \leq r\}  = \{\hat z_r'\leq 0  \} \cap \overline B$, and combining
 the last two equalities we then easily obtain the assertion.
\end{proof}

\section{Proofs for Section \ref{sec:meas-sep}}\label{sec:proof-2}

To prove Theorem  \ref{strong-measurable-and-int} we again need a couple of 
preliminary results. 
Most of these results consider, in one form or the other, 
 the following function
 $\Psi:I\to [0,\infty)$ defined by
\begin{equation}\label{def-spi}
 \Psi(r) :=   \inf_{z'\in S^+} \sup_{x\in \{\G=r\}} \bigl|\langle z', x\rangle\bigr|\, , \qquad \qquad r\in I,
\end{equation}
where $S^+ := \{ z'\in E':  \snorm{z'_{|H}}_{E'}=1 \mbox{ and }  \langle z',\xstar\rangle \geq 0\}$.

Our first result shows that the functionals found in Theorem \ref{unique-separation-char-2} are essentially the 
only minimizers of the outer infimum in \eqref{def-spi}.

\begin{lemma}\label{unique-minimizer}
Assume that \textbf{B1}, \textbf{B2*}, \textbf{B3}, \textbf{G1*}, and \textbf{G2} are satisfied.
 Then, for all $r\in I$, we have $\Psi(r) = 0$, and there exists a $z'\in S^+$ such that 
\begin{align}\label{psi-again}
 \Psi(r) = \sup_{x\in \{\G=r\}} |\langle z', x\rangle|\, .
\end{align}
Moreover, for every $z'\in S^+$ satisfying \eqref{psi-again}, we have the following implications
\begin{align*}
 \G(\xstar) <r \qquad \qquad & \Rightarrow \qquad \qquad z'_{|H} = -z_r'\\
\G(\xstar) =r \qquad \qquad & \Rightarrow \qquad \qquad z'_{|H} = \pm z_r'\\
\G(\xstar) >r \qquad \qquad & \Rightarrow \qquad \qquad z'_{|H} = z_r'\, ,
\end{align*}
where $(z_r')$ is the unique normalized separating family obtained in Theorem \ref{unique-separation-char-2}.
\end{lemma}

\begin{proof}
To show the existence of $z'\in S^+$, we assume without loss of generality that $\G(\xstar) \geq r$.
 Then the unique normalized separating functional $z_r'\in(H, \enorm\cdot)'$ found in Theorem \ref{unique-separation-char-2}
satisfies 
\begin{displaymath}
 \sup_{x\in \{\G=r\}} |\langle z_r', x\rangle| = 0\, ,
\end{displaymath}
 and since $\Psi(r)\geq 0$, we conclude that 
\begin{displaymath}
 \Psi(r) = \sup_{x\in \{\G=r\}} |\langle z_r', x\rangle| = 0\, .
\end{displaymath}
In addition, $\G(\xstar) \geq r$ implies $\langle z_r',\xstar\rangle \geq 0$. Extending $z_r'$ to a bounded
linear functional $z'\in E'$ with the help of Hahn-Banach's extension  theorem,
see e.g.~\cite[Theorem 1.9.6]{Megginson98}, then yields  $z'\in S^+$, and as a by-product of the proof, we have also 
established $\Psi(r) = 0$.

To show the implications, we restrict our considerations to the case $\G(\xstar) <r$, 
the remaining two cases can be treated analogously.
Then the already established $\Psi(r) = 0$ yields $\langle z', x\rangle = 0$ for all $x\in \{\G=r\}$, that is 
$\{\G=r\} \subset B\cap \ker z'$. Since $\snorm{z_r'}_{E'} = 1=\snorm{z'_{|H}}_{E'}$, we then conclude 
by Lemma \ref{equal-modulo-orientation} and Theorem \ref{unique-separation-char-2} that 
$z_r' = -z_{|H}'$ or $z_r' = z_{|H}'$. Assume that the latter is true. Then
$\G(\xstar) <r$ implies $0> \langle z_r', \xstar\rangle = \langle z',\xstar\rangle \geq 0$, and hence we have found 
a contradiction. Consequently, we have  $z_r' = -z_{|H}'$.
\end{proof}

Our next goal is to show that there exists a measurable selection of the minimizers of the function $\Psi$.
To this end, we first need to show that the inner supremum is measurable, and to show 
this, we
 now consider the functions $\Phi_n:I\times E'\to \R$, $n\in \N\cup\{\infty\}$ defined by 
\begin{equation}\label{def-phin}
 \Phi_n(r, z') := \sup_{x\in \{\G=r\} \cap nB_E} \bigl|\langle z', x\rangle\bigr|\, , \qquad \qquad (r,z')\in I\times E'\, ,
\end{equation}
where $I\subset \R$ is an interval and $E$ is a normed space.
The following lemma shows that $\Phi_n$ is continuous in the second variable.

\begin{lemma}\label{continuous-in-z}
Let $E$ be a normed space, $B\subset E$ be non-empty, and $\G:B\to \R$ be a continuous map. 
  Then, for all $n\in \N$ and $r\in I$, the map $\Phi_n(r,\mycdot):E'\to \R$ defined by \eqref{def-phin} is continuous.
\end{lemma}

\begin{proof}
 For $z_1', z_2'\in E'$ the triangle inequality for suprema yields
\begin{align*}
 \bigl| \Phi_n(r,z_1') - \Phi_n(r,z_2)\bigr|
& =\Biggl| \sup_{x\in \{\G=r\}\cap nB_E} \bigl|\langle z_1', x\rangle\bigr|  - \sup_{x\in \{\G=r\}\cap nB_E} \bigl|\langle z_2', x\rangle\bigr| \Biggr|\\
& \leq \sup_{x\in \{\G=r\}\cap nB_E} \bigl| \langle z_1',x\rangle - \langle z_2', x\rangle \bigr| \\
& \leq  \snorm{z_1'-z_2'}_{E'} \cdot n \, .
\end{align*}
Now the assertion easily follows.
\end{proof}

The next lemma shows that the function $\Phi_n$ is measurable in the first variable, provided that some technical 
assumptions are met.

\begin{lemma}\label{measurable-in-r}
Let $E$ be a separable Banach space, $B\subset E$ be non-empty, and $\G:B\to \R$ be a map satisfying \textbf{G4}.
%
Then, for all $n\in \N$ and $z'\in E'$, the map $\Phi_n(\mycdot, z'):I\to \R$ defined by \eqref{def-phin}
is $\hat{\ca B}(I)$-measurable.  
\end{lemma}

\begin{proof}
 Let us write $B_n := \G^{-1}(I)\cap nB_E$.
Note that  $nB_E$ is closed and thus $\ca B(E)$-measurable. Since $\G^{-1}(I)$
is $\ca B(E)$-measurable by \textbf{G4}, we conclude that $B_n$ is $\ca B(E)$-measurable.
Consequently,
 $\eins_{E\setminus B_n}:E\to \R$ is $\ca B(E)$-measurable, and
 the extension $\hat \G:E\to \R$ defined by 
\begin{displaymath}
 \hat \G(z) :=
\begin{cases}
 \G(z) & \mbox{ if } z\in B_n\\
0 & \mbox{ otherwise.}
\end{cases}
\end{displaymath}
is also $\ca B(E)$-measurable. Consequently, the map $h:I\times E\to \R^2$ defined by
\begin{displaymath}
 h(r,z) := \bigl(  \hat \G(z) - r,\, \eins_{E\setminus B_n}(z)\bigr)\, , \qquad \qquad (r,z)\in I\times E
\end{displaymath}
is ${\ca B}(I)\otimes \ca B(E)$-measurable. Moreover, note that the definition of $h$ yields 
\begin{displaymath}
 \{z\in E: h(r,z) = 0\} = \{z\in B_n: \G(z) = r\} = \{\G=r\}\cap nB_E\, .
\end{displaymath}
For $F:I\to 2^E$ defined by 
\begin{displaymath}
 F(r) :=  \bigl\{z\in E: h(r,z) \in \{ 0\}\bigr\}\, , 
\end{displaymath}
we thus find
$F(r) = \{\G=r\}\cap nB_E$ for all $r\in I$. 
Moreover, the graph of $F$, that is  
\begin{displaymath}
 \Gr F := \bigl\{  (r,z)\in I\times E: z\in F(r)    \bigr\}
 =  \bigl\{  (r,z)\in I\times E: h(r,z)=0   \bigr\}
\end{displaymath}
is ${\ca B}(I)\otimes \ca B(E)$-measurable, and 
 $\xi:I\times E\to \R$ defined by $\xi(r,z) :=  |\langle z', z\rangle|$
is continuous and thus $\ca B(I\times E)$-measurable.
Moreover, we have $\ca B(I\times E) = {\ca B}(I)\otimes \ca B(E)$ by 
\cite[Lemma 6.4.2]{Bogachev07_II} since $I$ and $E$ are both separable, and thus 
$\xi$ is 
${\ca B}(I)\otimes \ca B(E)$-measurable, too. Since separable Banach spaces are Polish spaces,
\cite[Lemma III.39 on p.~86]{CaVa77} then shows that the map 
\begin{displaymath}
 r\mapsto \sup_{z\in F(r)} \xi(r,z)
\end{displaymath}
is $\hat{\ca B}(I)$-measurable. From the latter we easily obtain the assertion.
\end{proof}

With the help of the two previous results, the 
 next result now establishes the desired measurability of $\P$.
Unfortunately, it requires a stronger separability assumption than the preceding lemmas.

\begin{corollary}\label{measurable-Phi}
Let $E$ be a  Banach space whose dual $E'$ is separable, $B\subset E$ be non-empty, 
and $\G:B\to \R$ be a continuous map satisfying \textbf{G4}.
%
%
Then  $\Phi_\infty:I\times E'\to \R$ is $\hat{\ca B}(I)\otimes \ca B(E')$-measurable.   
\end{corollary}

\begin{proof}
 Let us first recall, see e.g.~\cite[Theorem 1.10.7]{Megginson98}, that dual spaces are always Banach spaces. Consequently,
$E'$ is a Polish space. Moreover, the separability of $E'$ implies 
the separability of $E$, see e.g.~\cite[Theorem 1.12.11]{Megginson98}, and hence the map
$\Phi_n(\mycdot, z'):I\to \R$ is $\hat{\ca B}(I)$-measurable for all $z'\in E'$ and $n\in \N$ by Lemma \ref{measurable-in-r}.
Since  $\Phi_n(r,\mycdot):E'\to \R$ is continuous for all $r\in I$ and $n\in \N$ by Lemma \ref{continuous-in-z}, we conclude that 
$\Phi_n$ is a Carath\'eodory map. Moreover,  $E'$ is Polish,  
and thus $\Phi_n$ is $\hat{\ca B}(I)\otimes \ca B(E')$-measurable for all $n\in \N$, see e.g.~\cite[Lemma III.14 on p.~70]{CaVa77}.
Finally, we have $\Phi_\infty(r,z') = \lim_{n\to \infty} \Phi_n(r,z')$ for all $(r,z')\in I\times E'$, and hence 
$\Phi_\infty$ is also $\hat{\ca B}(I)\otimes \ca B(E')$-measurable.   
\end{proof}

The next result shows that we can find the minimizers of the infimum used in the definition
of $\Psi:I\to [0,\infty)$ in a measurable fashion.

\begin{theorem}\label{measurable-selection}
Assume that \textbf{B1}, \textbf{B2*}, \textbf{B3}, \textbf{B5}, \textbf{G1*}, \textbf{G2}, and \textbf{G4} are satisfied.
%
%
%
 Then there 
exists a measurable map $\zeta:(I, \hat{\ca B}(I))\to (E', \ca B(E'))$ such that, 
for all $r\in I$, we have $\zeta(r) \in S^+$ and 
\begin{displaymath}
 \Psi(r) =  \sup_{x\in \{\G=r\}} |\langle \zeta(r), x\rangle|\, .
\end{displaymath}
\end{theorem}

\begin{proof}
 Let us first show that $S^+$ is closed. To this end, we pick a sequence $(z_n')\subset S^+$ that converges 
in norm to some $z'\in E'$. Then $\langle z_n',\xstar\rangle \geq 0$ immediately implies $\langle z',\xstar\rangle\geq 0$.
To show that $\snorm{z'_{|H}}_{E'}=1$ we first observe that, for $x\in H$ with $\enorm x\leq 1$, we easily find
\begin{displaymath}
 |\langle z', x\rangle| = \lim_{n\to \infty} |\langle z_n', x\rangle| \leq 1\, ,
\end{displaymath}
and thus  $\snorm{z'_{|H}}_{E'}\leq 1$. To show the converse inequality, we pick, for all $n\geq 1$, 
an $x_n \in H$ with $\enorm {x_n} \leq 1$ such that  $1-1/n\leq |\langle z_n', x_n\rangle| \leq 1$. Then we obtain
\begin{displaymath}
 \bigl| \langle z', x_n\rangle - 1\bigr| \leq \bigl| \langle z'-z_n', x_n\rangle \bigr| + \bigl| \langle z_n', x_n\rangle -1\bigr|
\leq \snorm{z'-z_n'}_{E'} + 1/n\, ,
\end{displaymath}
and since the right hand-side converges to 0, we find $\snorm{z'_{|H}}_{E'}\geq 1$. Consequently, we have shown $z\in S^+$, 
and therefore, $S^+$ is indeed closed. From the latter, we conclude that $\eins_{E'\setminus S^+}:E'\to \R$ is 
$\ca B(E')$-measurable. Moreover, Corollary \ref{measurable-Phi} showed that 
$\Phi_\infty:I\times E'\to \R$
is  $\hat{\ca B}(I)\otimes \ca B(E')$-measurable, and consequently, the map $h:I\times E'\to \R^2$ defined by 
\begin{displaymath}
 h(r,z) := \bigl(\eins_{E'\setminus S^+}(z'), \, \Phi_\infty(r,z')  \bigr)\, , \qquad \qquad (r,z')\in I\times E',
\end{displaymath}
is also  $\hat{\ca B}(I)\otimes \ca B(E')$-measurable. We define $F:I\to 2^{E'}$ by
\begin{displaymath}
 F(r) := \bigl\{z'\in E': h(r,z') = 0\bigr\}\, , \qquad \qquad r\in I.
\end{displaymath}
Note that our construction ensures 
\begin{equation}\label{measurable-selection-h1}
 F(r) = \{z\in S^+:  \Phi_\infty(r,z') = 0\} = \biggl\{ z'\in S^+: \Psi(r) = \sup_{x\in \{\G=r\}} |\langle z', x\rangle|\biggr\}\, ,
\end{equation}
where in the last step we used the equality $\Psi(r) = 0$ established in Lemma \ref{unique-minimizer}. 
Moreover, the latter lemma also showed $F(r) \neq \emptyset$ for all $r\in I$, 
that is 
\begin{displaymath}
 \Dom F := \{r\in I: F(r) \neq \emptyset   \} = I\, .
\end{displaymath}
Since 
$E'$ is Polish, Aumann's
measurable selection principle, see \cite[part ii) of Lemma A.3.18]{StCh08} 
or \cite[Theorem III.22 on p.~74]{CaVa77}
yields a 
a measurable map $\zeta:(I, \hat{\ca B}(I))\to (E', \ca B(E'))$ with $\zeta(r) \in F(r)$ for all $r\in I$. 
Then \eqref{measurable-selection-h1} shows that $\zeta$ is the desired map.
\end{proof}

With these preparations, we can finally prove Theorem \ref{strong-measurable-and-int}.
The basic idea behind this proof is to combine Lemma \ref{unique-minimizer}
and Theorem \ref{measurable-selection}.

\begin{proof}[Proof of Theorem \ref{strong-measurable-and-int}]
%
Let us now consider the measurable selection 
$\zeta:I\to E'$ from Theorem \ref{measurable-selection}. 
Furthermore, we  fix an $r\in I$.
If $r> \G(\xstar)$, then Lemma \ref{unique-minimizer} shows that $\zeta(r)_{|H} = - z_r'$, and thus
$\zeta(r) = -\hat z_r'$ by \textbf{B4}.  Analogously, $r< \G(\xstar)$
implies $\zeta(r) = \hat z_r'$, and in the case $r=\G(\xstar)$ we have either 
$\zeta(r) = -\hat z_r'$ or $\zeta(r) = \hat z_r'$. From these relations 
it is easy to obtain the desired measurability of $Z:(I, \hat{\ca B}(I))\to (E', \ca B(E'))$.

Since the image $Z(I)$ is separable
by the separability of $E'$, we further see by \cite[Theorem 8, page 5]{Dinculeanu00} that $Z$ is an $E'$-valued 
measurable function in the sense of Bochner 
integration theory. 
\end{proof}

\begin{proof}[Proof of Corollary \ref{meas-cor}]
   We first need to verify the remaining assumptions of Theorem \ref{strong-measurable-and-int} for $E_0$.
   To this end, we first observe that $\p'_{|E_0}\in E_0'$ and therefore \textbf{B1} is satisfied
   for  $\p'_{|E_0}$.
   Moreover,  \textbf{B2*}
   and \textbf{G2} 
   are independent of $E_0$, and hence they are also satisfied. Consequently, we indeed 
    obtain a family $(\hat z_{0,r}')_{r\in I}\subset E_0'$ of separating functionals by 
Theorem \ref{unique-separation-char-2} and this family is measurable in the sense of Theorem \ref{strong-measurable-and-int}
with respect to the space $E_0'$. 

Now, Theorem \ref{unique-separation-char-2} yields for both families and all $r\in I$ that 
\begin{align*}
   \{\G=r\} &= \ker (\hat z_r')_{|H} \cap B \\
    \{\G=r\} &= \ker (\hat z_{0,r}')_{|H} \cap B \, ,
\end{align*}
and consequently, we obtain an $\a(r) \neq 0$ such that 
\begin{align}\label{meas-cor-p1}
    (\hat z_r')_{|H} = \a(r) (\hat z_{0,r}')_{|H}
\end{align}
by Lemma \ref{equal-modulo-orientation}. By fixing an $x\in  \{\G>r\}$ we further see by Theorem \ref{unique-separation-char-2}
that both functionals have the same orientation, and thus we find $\a(r) >0 $. In addition, \eqref{meas-cor-a1}
easily follows by the denseness of $H$ in $E_0$. To show that $\a$ is measurable, we first 
recall that we have $H\subset E_0\subset E$, and since $H$ is dense in $E$, we see that  $E_0$ is dense in $E$, too.
Moreover, $E_0$ is separable by \textbf{B5} and therefore we conclude that $E$ is separable. Consequently, there 
exists an at most countable $D\subset H$ such that $D\subset B_E$ is dense. 
Moreover, \eqref{meas-cor-p1} shows that $(\hat z_{0,r}')_{|H}$ is also continuous with respect to $\enorm\cdot$
and therefore, we obtain 
\begin{displaymath}
   h(r):= \mnorm{(\hat z_{0,r}')_{|H}}_{(H,\enorm\cdot)'}
   = 
   \sup_{x\in H\cap B_E} \bigl|\langle \hat z_{0,r}', x \rangle\bigr|
   = 
   \sup_{x\in D} \bigl|\langle \hat z_{0,r}', x \rangle\bigr|\, .
\end{displaymath}
Now, for each $x\in D$, Theorem \ref{strong-measurable-and-int} shows that $r\mapsto \langle \hat z_{0,r}', x \rangle$
is measurable with respect to the $\s$-algebras $\hat {\ca B}(I)$ and $\ca B(\R)$, and therefore $r\mapsto h(r)$ inherits this 
measurability. Moreover, using $\snorm{(\hat z_r')_{|H}}_{E'} =1$ and \eqref{meas-cor-p1} we find
$1 = \a(r) h(r)$ for all $r\in I$, and from the latter we easily obtain the desired measurability of $\a$.
\end{proof}

\section{Proofs for Section \ref{sec:cont-sep}}\label{sec:proof-3}

\begin{lemma}\label{cont-in-kernel}
 Let \textbf{B1}, \textbf{B2*}, \textbf{G1}, and \textbf{G2} be satisfied, and $(z_r')_{r\in I}\subset H'$
 be the unique normalized family of separating functionals obtained in Theorem \ref{unique-separation-char-1}.
 Then, for all  $r_0\in I$ and $z\in \ker z_{r_0}'$,  we have
 \begin{displaymath}
    \lim_{r\to r_0} \langle z_r', z\rangle = 0\, .
 \end{displaymath}
 Moreover, if \textbf{B3} and \textbf{G1*} are additionally satisfied, then the same holds for the 
 unique functionals $z_r'\in (H, \enorm\cdot)'$ obtained in Theorem \ref{unique-separation-char-2}.
\end{lemma}

\begin{proof}
 Let us first consider the case $z\in \{\G=r_0\}$. For $\e>0$  there then exist $x^-\in \{\G<r_0\}$ and 
$x^+\in \{\G>r_0\}$ such that $\hnorm{z-x^-}\leq \e$ and $\hnorm{z-x^+}\leq \e$. Consequently,
there exists a $\d>0$ such that $[r_0-\d, r_0+\d] \subset [\G(x^-), \G(x^+)]$. For $\a\in [0,1]$
we define $x_\a := (1-\a)x^- + \a x^+$. Clearly, this gives $\hnorm{z-x_\a} \leq \e$ for all $\a\in [0,1]$.
Moreover, the $\fnorm\cdot$-continuity of $\G$ together with the intermediate theorem 
 shows that, for all $r\in (r_0-\d, r_0+\d)$ there exists 
an $\a_r\in [0,1]$  such that $\G(x_{\a_r}) = r$, that is $x_{\a_r}\in \{\G=r\} \subset \ker z_r'$. 
For $r\in (r_0-\d, r_0+\d)$, 
this yields
\begin{displaymath}
 \bigl| \langle z_r', z\rangle \bigr|
 \leq   \bigl| \langle z_r', z - x_{\a_r} \rangle \bigr| +  \bigl| \langle z_r', x_{\a_r}\rangle \bigr|
 \leq \hnorm{z- x_{\a_r}}   \leq \e\, .
\end{displaymath}
This shows the assertion for $z\in \{\G=r_0\}$.
The general case $z\in \ker z_{r_0}'$ now follows from 
$ \ker z_{r_0}' = \spann (\ker z_{r_0}'\cap B) = \spann \{\G=r_0\}$ 
established in Lemma \ref{kernel-is-determined}.

Finally, if \textbf{B3} and \textbf{G1*} are   satisfied and $(z_r')\subset (H, \enorm\cdot)'$ 
denotes the unique  separating functionals obtained by Theorem \ref{unique-separation-char-2}
we can literally repeat the first part 
for $z\in \{\G=r_0\}$ and obtain 
\begin{align*}
 \bigl| \langle z_r', z\rangle \bigr|
 \leq   \bigl| \langle z_r', z - x_{\a_r} \rangle \bigr| +  \bigl| \langle z_r', x_{\a_r}\rangle \bigr|
   \leq \enorm{z- x_{\a_r}}
 \leq \hnorm{z- x_{\a_r}}   \leq \e\, .
%
\end{align*}
by Lemma \ref{F-is-in-kernel}.
The general case $z\in \ker z_{r_0}'$ again follows by Lemma \ref{kernel-is-determined}.
\end{proof}

\begin{lemma}\label{weak-convergence}
  Let \textbf{B1}, \textbf{B2*}, \textbf{B3}, \textbf{B4}, \textbf{B6}, \textbf{G1*}, and \textbf{G2} be satisfied, and 
  $(\hat z_r')_{r\in I}\subset E'$
 be the family of separating functionals obtained in Theorem \ref{unique-separation-char-2}. Moreover, let 
  $r\in I$ and  $(r_n)\subset I$ with $r_n\to r$. Then there exist a subsequence 
  $(r_{n_k})$ of $(r_n)$ and an $\a\in [0,1]$ such that for all $z\in E$ we have 
  \begin{displaymath}
   \langle \hat z'_{r_{n_k}}, z\rangle \to \langle \a \hat z'_r, z\rangle \, . 
  \end{displaymath}
\end{lemma}

\begin{proof}
 Since $(\hat z_{r_n}')_{r\in I}\subset B_{E'}$, 
  the sequential Banach-Alaoglu theorem, see e.g.~\cite[Theorem 2.6.18 in combination with Theorem 2.6.23]{Megginson98}
  and also \cite[Exercise 2.73]{Megginson98},
  guarantees that there exist a subsequence
  $(\hat z_{r_{n_k}}')$ and an $z'\in B_{E'}$ such that 
  \begin{equation}\label{weak-convergence-p1}
   \langle \hat z'_{r_{n_k}}, z\rangle \to \langle z', z\rangle  
  \end{equation}
  for all $z\in E$. By Lemma \ref{cont-in-kernel} we conclude that $\langle z', z\rangle = 0$
  for all $z\in \ker \hat z_r'$, and thus $\ker (\hat z_r')_{|H} \subset \ker z'_{|H}$.
  Consequently, we have both $\{\G=r\} = B\cap \ker (\hat z_r')_{|H}$ and  $\{\G=r\} \subset B\cap \ker z'_{|H}$ 
  and therefore Lemma \ref{equal-modulo-orientation} gives an $\a\in \R$ such that 
  $z'_{|H}= \a \cdot (\hat z_r')_{|H}$, and thus $z' = \a\hat z_r'$ by the denseness of $H$ in $E$.
  Using $\snorm{z'} \leq 1 = \snorm{\hat z_r'}$ we find $\a\in [-1,1]$ and 
  \eqref{weak-convergence-p1} gives the desired convergence. Let us finally show that $\a\geq 0$. To this end,
  note that by Lemma \ref{quasi-monotone-level-sets} we find
    an $r_0\in I$ with $r_0> r_n$ for all $n\geq 1$. This obviously gives $r_0\geq r$.
    Let us further fix a $z\in \{\G=r_0\}$. Then we have $z\in \{\G>r_n\} = \{z_n' >0\}\cap B$ and 
    thus $\langle \hat z'_{r_{n_k}}, z\rangle > 0$ for all $k\geq 1$. Analogously we conclude from $r_0\geq r$ that 
    $\langle \hat z'_r, z\rangle \geq 0$, and thus $\a<0$ is impossible.
\end{proof}

\begin{lemma}\label{margin-calculation}
 Let $E$ be an arbitrary normed space. Then for all $x\in E$ and all $x'\in {E'}$ with $\snorm{x'}_{E'}=1$ we have 
\begin{displaymath}
 d(x, \ker x') = |\langle x', x\rangle|\, .
\end{displaymath}
\end{lemma}

\begin{proof}
``$\leq$'': Let us fix an $\e\in (0,1)$. Since $\snorm{x'}=1$, there then exists an $x_0\in B_E$ with 
$\langle x', x_0\rangle \geq 1-\e$. Clearly, this gives $x_0\not\in \ker x'$, and an easy calculation then 
shows 
\begin{displaymath}
 z:= x - \frac{\langle x',x\rangle}{\langle x',x_0\rangle}x_0 \in \ker x'\, .
\end{displaymath}
From the latter we then conclude that 
\begin{displaymath}
 d(x,\ker x') \leq \enorm{x-z} = \biggl| \frac{\langle x',x\rangle}{\langle x',x_0\rangle} \biggr| \cdot \enorm{x_0}
\leq \frac{|\langle x', x\rangle|}{1-\e}\, .
\end{displaymath}
Letting $\e\to 0$, then gives the desired inequality.

``$\geq$'': If $x\in \ker x'$, there is nothing to prove, and hence we assume without
loss of generality that $x\not \in \ker x'$. For $\e>0$, we now fix an $z\in \ker x'$ such that 
$\enorm{x-z} \leq d(x,\ker x') + \e$. Then $x\not \in \ker x'$ ensures $x-z\not \in \ker x'$, and thus we find
\begin{displaymath}
d(x,\ker x') + \e\geq  \enorm{x-z} 
= \biggl| \frac{\langle x',x\rangle}{\langle x',x-z\rangle} \biggr| \cdot \enorm{x-z} 
\geq |\langle x',x\rangle|\, ,
\end{displaymath}
where in the last step we used $|\langle x',x-z\rangle| \leq \enorm{x-z}$.
Letting $\e\to 0$, then gives the desired inequality.
\end{proof}

\begin{proof}[Proof of Theorem \ref{cont-dep}]
  \atob i {iii}
   By Lemma \ref{weak-convergence} is suffices to show that 
   independent of the sequence $(r_n)$ and its subsequence 
   we always have $\a=1$.
   To show the latter let us first assume that \eqref{g5-conv} is actually satisfied for some $x\in B \setminus \spann\{\G=r  \}$.
   Let us assume without loss of generality that $r_0:= \G(x)$ satisfies $r_0>r$. Then there is an $n_0\geq 1$ such that
    $r_0 > r_n$ for all 
   $n\geq n_0$ and thus we find both $x\in \{\G>r_n\} \subset \{z_{r_n}' >0\}$
   and $x\in \{\G>r\} \subset \{z_{r}' >0\}$.
   Combining Lemma \ref{margin-calculation} with Lemma \ref{kernel-is-determined} and Theorem \ref{unique-separation-char-2}
   we then find 
   \begin{align}\label{cont-dep-p1}
       \langle \hat z'_{r_{n}}, x\rangle 
       = 
       d(x, \ker z_{r_n}')
       =
        d\bigl(x, \spann\{\G=r_n  \}  \bigr) 
        \to  d\bigl(x, \spann\{\G=r  \}  \bigr)
       = 
        d(x, \ker z_{r}')
        =
       \langle  \hat z'_r, x\rangle 
   \end{align}
    and since $\langle  \hat z'_r, x\rangle >0$   we conclude that we indeed always have $\a=1$.
    
    In the remaining part of the proof we show that the strong version of  \eqref{g5-conv} used above is implied by \textbf{G5}.
    To this end, let us first  assume that \eqref{g5-conv}  is satisfied for some $x\in F \setminus \spann\{\G=r  \}$.
    By 
    Lemma \ref{all-interirors}  we   fix 
      an $x_r\in \{\G=r\}$ such that $-\xstar+x_r\in\relint AF$.  Consequently, there exists an $\e>0$ such that 
      for all $y\in F$ with $\fnorm y\leq \e$ we have $-\xstar+x_r + y \in A =  -\xstar + B$, that is 
      $x_r + y \in  B$.  Moreover, we easily find a $\d>0$ such that $\fnorm{\d x}\leq \e$ and thus 
      we obtain $\bar x := x_r  + \d x\in B$. In addition, $x\not\in \spann\{\G=r  \}$  together with 
      $x_r\in \spann\{\G=r  \}$ yields $\bar x\not\in \spann\{\G=r  \}$.
      Let us verify that \eqref{g5-conv} holds for $\bar x$.
      To this end, we assume without loss of generality that $\langle z_r',x \rangle > 0$.
      Repeating the arguments in \eqref{cont-dep-p1} we then find 
      $| \langle \hat z'_{r_{n}}, x\rangle | \to  \langle  \hat z'_r, x\rangle $, and by Lemma \ref{weak-convergence}
      we see that $-\langle \hat z'_{r_{n_k}}, x\rangle  \to  \langle  \hat z'_r, x\rangle $ for some subsequence $(r_{n_k})$  is 
      impossible. Therefore, we actually have $\langle \hat z'_{r_{n}}, x\rangle  \to  \langle  \hat z'_r, x\rangle $.
%
      In addition, Lemma \ref{cont-in-kernel} shows that $\langle   z'_{r_{n}}, x_r\rangle \to 0$.
      With these preparatory considerations we now obtain, analogously to  \eqref{cont-dep-p1}, that 
      \begin{align*}
          d\bigl(\bar x, \spann\{\G=r_n  \}  \bigr) 
          = \bigl| \langle z_{r_n}',    x_r  + \d x     \rangle       \bigr|
          \to 
          \bigl| \langle z_{r}',    x_r  + \d x     \rangle       \bigr|
          =
           d\bigl(\bar x, \spann\{\G=r  \}  \bigr) \, ,
      \end{align*}
      that is $\bar x\in B \setminus \spann\{\G=r  \}$ satisfies \eqref{g5-conv}. 
      
      In our last step, we assume that only \textbf{G5}  is satisfied, i.e.~\eqref{g5-conv}
      holds for some $x\in H \setminus \spann\{\G=r  \}$.
      With the help of the previous step, it then suffices to find an $y \in F \setminus \spann\{\G=r  \}$
      for which \eqref{g5-conv} holds. To this end, recall that Lemma \ref{generating-H} showed $H= F\oplus \R x_r$,
      where $x_r \in \{\G=r\}$ is again a vector satisfying $-\xstar+x_r\in\relint AF$. 
      Consequently, we have $x = y + \a x_r$, for some suitable $y\in F$ and $\a\in \R$, and since 
      we have already considered the case $\a=0$ in the previous step, we may assume that $\a\neq 0$.
      Now, we clearly have $y = x - \a x_r \in F$, and since $x_r\in  \spann\{\G=r  \}$ but 
      $x\not \in  \spann\{\G=r  \}$, we find $y\not\in \spann\{\G=r  \}$, that is $x\in F\setminus \spann\{\G=r  \}$.
      Finally, verifying \eqref{g5-conv} for $y$ is analogous to the previous case.
      
      \atob {iii} {ii} This immediately follows from $ \ker z' = \spann \{\G=r\}$, which has been  established in 
      Lemma \ref{kernel-is-determined}, and Lemma \ref{margin-calculation}.
      
      \atob {ii} i This implication is trivial.
\end{proof}

\begin{proof}[Proof of Corollary \ref{cont-dep-fin-dim}]
   Let 
  $r\in I$ and  $(r_n)\subset I$ with $r_n\to r$. Since $E$ is separable, Lemma \ref{weak-convergence} shows that 
   there exist a subsequence 
  $(r_{n_k})$ of $(r_n)$ and an $\a\in [0,1]$ such that for all $z\in E$ we have 
  \begin{displaymath}
   \langle \hat z'_{r_{n_k}}, z\rangle \to \langle \a \hat z'_r, z\rangle \, . 
  \end{displaymath}
  Moreover, in finite dimensional spaces, 
  weak*-convergence implies norm-convergence, and hence we obtain $\snorm{\hat z'_{r_{n_k}} -\a\hat z'_r }_{E'}\to 0$.
  Since $\snorm{\hat z'_{r_{n_k}} }_{E'} = 1 = \snorm{\hat z'_{r} }_{E'}$, we then find $\a=1$, and since 
  the sequence $(r_n)\subset I$ was chosen arbitrarily, we obtain the assertion by a standard argument.
\end{proof}

\subsubsection*{Acknowledgment}

Much of the work of this paper was done while I
was visiting the NICTA research lab in Canberra.  NICTA is funded by the
Australian Government. Special thanks go to Bob Williamson who made this stay possible, 
who   introduced me to the question of elicitation, and with whom I had many fruitful discussions
about this subject.

{\small
\bibliographystyle{plain}
\bibliography{../../literatur-DB/steinwart-mine,../../literatur-DB/steinwart-books,../../literatur-DB/steinwart-proc,../../literatur-DB/steinwart-article}

\begin{thebibliography}{10}

\bibitem{AgAg15a}
A.~Agarwal and S.~Agarwal.
\newblock On consistent surrogate risk minimization and property elicitation.
\newblock In P~Gr{\"u}nwald and E.~Hazan, editors, {\em JMLR Workshop and
  Conference Proceedings Volume 40: Proceedings of the 28th Conference on
  Learning Theory 2015}, pages 1--19, 2015.
\newblock \url{http://www.jmlr.org/proceedings/papers/v40/Agarwal15.pdf}.

\bibitem{Beauzamy82}
B.~Beauzamy.
\newblock {\em Introduction to {B}anach spaces and their geometry}.
\newblock North-Holland Publishing Co., Amsterdam-New York, 1982.

\bibitem{Bogachev07_II}
V.~I. Bogachev.
\newblock {\em Measure Theory, Vol.~II}.
\newblock Springer-Verlag, Berlin, 2007.

\bibitem{CaVa77}
C.~Castaing and M.~Valadier.
\newblock {\em Convex Analysis and Measurable Multifunctions}.
\newblock Springer-Verlag, Berlin, 1977.

\bibitem{Dinculeanu00}
N.~Dinculeanu.
\newblock {\em Vector Integration and Stochastic Integration in {B}anach
  Spaces}.
\newblock John Wiley {\&} Sons, New York, 2000.

\bibitem{FrKa14a}
R.~Frongillo and I.~Kash.
\newblock General truthfulness characterizations via convex analysis.
\newblock In Tie-Yan Liu, Qi~Qi, and Yinyu Ye, editors, {\em Web and Internet
  Economics}, volume 8877 of {\em Lecture Notes in Computer Science}, pages
  354--370. 2014.

\bibitem{FrKa15a}
R.~Frongillo and I.~Kash.
\newblock Vector-valued property elicitation.
\newblock In P~Gr{\"u}nwald and E.~Hazan, editors, {\em JMLR Workshop and
  Conference Proceedings Volume 40: Proceedings of the 28th Conference on
  Learning Theory 2015}, pages 1--18. 2015.
\newblock \url{http://www.jmlr.org/proceedings/papers/v40/Frongillo15.pdf}.

\bibitem{Gneiting11a}
T.~Gneiting.
\newblock Making and evaluating point forecasts.
\newblock {\em J. Amer. Statist. Assoc.}, 106:746--762, 2011.

\bibitem{GnRa07a}
T.~Gneiting and A.~E. Raftery.
\newblock Strictly proper scoring rules, prediction, and estimation.
\newblock {\em J. Amer. Statist. Assoc.}, 102:359--378, 2007.

\bibitem{GrPi71a}
H.~J. Greenberg and W.~P. Pierskalla.
\newblock A review of quasi-convex functions.
\newblock {\em Oper. Res.}, 19:1553--1570, 1971.

\bibitem{LaPeSh08a}
N.~Lambert, D.~Pennock, and Y.~Shoham.
\newblock Eliciting properties of probability distributions.
\newblock In {\em Proceedings of the ACM Conference on Electronic Commerce},
  pages 129--138, 2008.

\bibitem{LambertXXa}
N.~S. Lambert.
\newblock Elicitation and evaluation of statistical forecasts.
\newblock Technical report, Stanford Graduate School of Business, 2013.
\newblock \url{http://web.stanford.edu/~nlambert/papers/elicitation.pdf}.

\bibitem{LiTz79}
J.~Lindenstrauss and L.~Tzafriri.
\newblock {\em Classical {B}anach spaces. {II}}.
\newblock Springer-Verlag, Berlin-New York, 1979.

\bibitem{Megginson98}
R.~E. Megginson.
\newblock {\em An Introduction to {B}anach Space Theory}.
\newblock Springer-Verlag, New York, 1998.

\bibitem{Osband85a}
K.~H. Osband.
\newblock {\em Providing Incentives for Better Cost Forecasting}.
\newblock PhD thesis, University of California, Berkeley, 1985.

\bibitem{ScSm02}
B.~Sch{\"o}lkopf and A.~J. Smola.
\newblock {\em Learning with Kernels}.
\newblock MIT Press, Cambridge, MA, 2002.

\bibitem{Steinwart07a}
I.~Steinwart.
\newblock How to compare different loss functions.
\newblock {\em Constr. Approx.}, 26:225--287, 2007.

\bibitem{StCh08}
I.~Steinwart and A.~Christmann.
\newblock {\em Support Vector Machines}.
\newblock Springer, New York, 2008.

\bibitem{StPaWiZh14a}
I.~Steinwart, C.~Pasin, R.~Williamson, and S.~Zhang.
\newblock Elicitation and identification of properties.
\newblock In M.~F. Balcan and C.~Szepesvari, editors, {\em JMLR Workshop and
  Conference Proceedings Volume 35: Proceedings of the 27th Conference on
  Learning Theory 2014}, pages 482--526, 2014.

\bibitem{ThPa73a}
W.~A. Thompson and D.~W. Parke.
\newblock Some properties of generalized concave functions.
\newblock {\em Operations Res.}, 21:305--313, 1973.

\bibitem{Vapnik98}
V.~N. Vapnik.
\newblock {\em Statistical Learning Theory}.
\newblock John Wiley {\&} Sons, New York, 1998.

\bibitem{WaZi15a}
R.~Wang and J.~F. Ziegel.
\newblock Elicitable distortion risk measures: A concise proof.
\newblock {\em Statist. Probab. Lett.}, 100:172--175, 2015.

\bibitem{Ziegel14a}
J.~F. Ziegel.
\newblock Coherence and elicitability.
\newblock {\em Math. Finance}, to appear, 2014.

\end{thebibliography}
}

\newpage

\end{document}